\documentclass[12pt,leqno,oneside]{amsart}
\usepackage{mathrsfs,dsfont}
\usepackage{amsmath,amstext,amsthm,amssymb,bbm}
\usepackage{charter}
\usepackage{typearea}
\usepackage{hyperref}
\usepackage{color}

\usepackage{float}
\usepackage{tikz}
\usetikzlibrary{patterns}
\usepackage{graphicx}
\usepackage{hyperref}

\usepackage[width=6.4in,height=8.5in]
{geometry}

\newtheorem{Theorem}{Theorem}[section]
\newtheorem{Proposition}[Theorem]{Proposition}
\newtheorem{Lemma}[Theorem]{Lemma}
\newtheorem{Corollary}[Theorem]{Corollary}
\theoremstyle{definition}

\newtheorem{Example}[Theorem]{Example}

\numberwithin{equation}{section}

\DeclareMathOperator{\supp}{supp}

\DeclareMathOperator{\Prob}{Prob}
\newcommand{\cali}[1]{\mathscr{#1}}
 
\newcommand{\cT}{\cali{T}}

\newcommand{\field}[1]{\mathbb{#1}}

\newcommand{\R}{\field{R}}
\newcommand{\C}{\field{C}}

\newcommand{\E}{\mathbb{E}}
\newcommand{\mO}{\mathcal{O}}

\newcommand{\hp}{H^0_{(2)}(X,L_p)}

\newcommand{\FS}{{{_\mathrm{FS}}}}

\newcommand{\comment}[1]{}

\begin{document}

\title{A survey on zeros of random holomorphic sections}

\author{Turgay Bayraktar} 
\thanks{T.\ Bayraktar is partially supported by T\"{U}B\.{I}TAK B\.{I}DEB-2232/118C006}
\address{Faculty of Engineering and Natural Sciences, Sabanc{\i} University, \.{I}stanbul, Turkey}
\email{tbayraktar@sabanciuniv.edu}

\author{Dan Coman}
\thanks{D.\ Coman is partially supported by the NSF Grant DMS-1700011}
\address{Department of Mathematics, 
Syracuse University, Syracuse, NY 13244-1150, USA}
\email{dcoman@syr.edu}

\author{Hendrik Herrmann}
\address{Univerisit\"at zu K\"oln, Mathematisches institut,
Weyertal 86-90, 50931 K\"oln, Germany}
\email{post@hendrik-herrmann.de}

\author{George Marinescu}
\address{Univerisit\"at zu K\"oln, Mathematisches institut,
Weyertal 86-90, 50931 K\"oln, Germany 
\newline\mbox{\quad}\,Institute of Mathematics `Simion Stoilow', 
Romanian Academy, Bucharest, Romania}
\email{gmarines@math.uni-koeln.de}
\thanks{H.\ Herrmann and G.\ Marinescu are partially supported by the DFG funded project CRC/TRR 191.}
\thanks{The authors were partially funded through the Institutional Strategy 
of the University of Cologne within the German Excellence Initiative (KPA QM2)}


\date{July 9, 2018}
\dedicatory{Dedicated to our friend Norm Levenberg 
on the occasion of his 60th birthday}

\begin{abstract}
We survey results on the distribution of zeros
of random polynomials and of random holomorphic 
sections of line bundles, especially for large classes 
of probability measures
on the spaces of holomorphic sections. 
We provide furthermore some new examples of measures 
supported in totally 
real subsets of the complex probability space.
\end{abstract}

\maketitle

\tableofcontents

%%%%%%%%%%%%%%%%%%%%%%%%%%%%%%
\section{Introduction}
The main purpose of this paper is to review some 
results on the distribution of zeros
of random polynomials and more generally of random holomorphic sections 
of line bundles.
A general motivating question is the following. 
If the coefficients of a polynomial are subject to random error, 
the zeros of
this polynomial will also be subject to random error. 
It is natural to enquire how the latter errors depend upon the former.
One considers therefore polynomials whose coefficients
are independent identically distributed random variables and studies the
statistical properties of the zeros, such as the number 
of real zeros and uniformity of the zero distribution.
Many classical works are devoted to this circle of ideas: 
Bloch-P{\'o}lya \cite{BPo},
Littelwood-Offord \cite{LO,LO1,LO2}, Erd\H{o}s-Tur{\'a}n \cite{ET}, 
Kac \cite{Kac}, Hammersley \cite{Ham56}.

The distribution of zeros of polynomials with random coefficients
is also relevant for problems which naturally arise
in the context of quantum chaotic dynamics or in other domains 
of physics \cite{BBH96,NoVo:98}.
One can view zeros of polynomials in one complex variable as 
interacting particles in two dimensions, as, for instance, 
eigenvalues of random asymmetric matrices can be physically interpreted 
as a two-dimensional electron gas confined in a disk.

There is an interesting connection between equidistribution of 
zeros and Quantum Unique Ergodicity related to a conjecture of 
Rudnik-Sarnak \cite{RuSa:94} about the behavior of 
high energy Laplace eigenfunctions on a Riemannian manifold. 
By replacing Laplace eigenfunctions with modular forms 
one is lead to study of the equidistribution of zeros of 
Hecke modular forms, see
Holowinsky-Soundararajan \cite{HolSou:10}.

The first purpose of this survey is to review some results on 
asymptotic equilibrium distribution of zeros of polynomials
which arose from the work of Bloom \cite{Bl05,Bl09} and Bloom-Levenberg
\cite{BL15}. They pointed out the role of
the extremal plurisubharmonic functions in the equidistribution result.
Bloom \cite{Bl05,Bl09} also introduced the Bernstein-Markov 
measures as a convenient 
general framework for defining the Gaussian random sections.
A conceptual and very fruitful approach
is to introduce an inner product on the space of polynomials
which induces the Gaussian probability measure and to
consider the asymptotics of the Bergman kernel associated to this 
inner product \cite{Bl05,Bl09,BL15,ShZ03}.
%
%To illustrate this, recall that a seminal result of Hammersley \cite{Ham56}
%states that the zeros of random complex Kac polynomials
%$f(z)=\sum_{j=0}^p a_jz^j$, $z\in\C$,
%tend to concentrate on the unit circle $S^1\in\C$
%as $p\to\infty$ when the coefficients $a_j$ are
%independent complex Gaussian random variables of mean 
%zero and variance one.
%For this purpose we identify the space $\C_p[z]$ of polynomials in $z$
%with degree less than $p$ with the space $\C^{p+1}$ by means of the
%basis $1, z,\ldots,z^p$ and endow $\C_p[z]$ with a
%Gaussian probability measure coming from $\C^{p+1}$.
%
%A conceptual and very fertile manner to prove the above result 
%\cite{BL15,ShZ03} is to
%interpret the Gaussian measure on the space $\C_p[z]$ 
%as that introduced by the inner product 
%$(f,g)=\frac{1}{2\pi}\int_{S^1}f(e^{i\theta})
%\overline{g(e^{i\theta})}d\theta$ on the unit circle $S^1$
%for which an orthonormal basis is $1, z,\ldots,z^p$.
%The key of the proof is the fact that the Bergman kernel
%$K_p(z,z)$ of $(\C_p[z],(\cdot,\cdot))$ satisfies 
%$\frac{1}{2p}\log K_p(z,z)\to\log^{+}|z|$ locally uniformly on $\C$
%as $p\to\infty$
%and $\frac{1}{2\pi}\log^{+}|z|$ is a potential of the delta function 
%$\delta_{S^1}$, which is also the equilibrium measure of the unit circle in $\C$.
This is generalized in Theorem \ref{BL2}
for the weighted global extremal function of a locally regular weighted closed set.
We will moreover study the distribution of common zeros of
a $k$-tuple of polynomials in Section \ref{S:rpm}
and give a central limit theorem for the linear statistics of the zeros in Section
\ref{S:clt}. 
 
In the second part of the paper we consider random holomorphic 
sections of line bundles over complex manifolds with respect
to weights and measures satisfying some quite weak conditions.
Polynomials of degree at most $p$ on
$\C$ generalize to the space $H^0(X, L^p)$ of holomorphic sections of the $p$-th 
power of an ample line bundle $L\to X$ over any complex manifold of dimension $n$. 
The weight  used in the case of polynomials generalizes
to a Hermitian metric $h$ on $L$. In the case $h$ is smooth
and has positive curvature $\omega=c_1(L,h)$, Shiffman-Zelditch \cite{ShZ99}
showed that the asymptotic equilibrium distribution 
(see also \cite{NoVo:98} for genus one
surfaces in dimension one) of zeros of sections in $H^0(X,L^p)$ as
$p\to\infty$ is given by the volume form $\omega^n/n!$
of the metric. Dinh-Sibony \cite{DS06} introduced another
approach, which also gives an estimate of the speed of convergence
of zeros to the equilibrium distribution (see also \cite{DMS} for the
non-compact setting).
This result was generalized for singular metrics
whose curvature is a K\"ahler current in \cite{CM11} and for
sequences of line bundles over normal complex spaces in \cite{CMM}
(see also \cite{CM13,CM13b,DMM}).
In Section \ref{S:equidist} we review results from \cite{BCM}
where we generalize the setting of \cite{CMM} for probability
measures satisfying a very general moment condition (see Condition (B)
therein), including
measures with heavy tail and small ball probability. 
Important examples are provided by the Guassians and the Fubini-Study
volumes. 

In Section \ref{S:mtrs}, we provide some new examples of measures 
that satisfy the moment condition. They have support contained in totally 
real subsets of the complex probability space.

Finally, in Section \ref{S:pictures} we illustrate the 
equilibrium distribution of zeros by pictures of the zero divisors for the case
of polynomials with inner product over the square in the plane
and for $SU_2$ polynomials.

%----------------------------------------------------
\section{Random Polynomials on $\mathbb{C}^n$}\label{S:I}
In this section we survey some statistical properties 
of zeros of random polynomials associated with a locally regular set 
$Y\subset\C^n$ and a weight function $\varphi:Y\to\mathbb{R}$. 
Classical Weyl polynomials arise as a special case. 
We denote by $\lambda_{2n}$ the Lebesgue measure on 
$\C^n\simeq\R^{2n}$. We also denote by $\C[z]$ 
the space of polynomials in $n$ complex variables, 
and we let $\C_p[z]=\{f\in\C[z]:\,\deg f\leq p\}$. 
Recall that $d=\partial+\overline{\partial}$ and 
$d^c:=\frac{i}{2\pi}(\overline{\partial}-\partial)$, 
so that $dd^c=\frac{i}{\pi}\partial\overline{\partial}$.

\subsection{Weighted global extremal function}
Let $Y\subset \C^n$ be a (possibly) unbounded  closed set and 
$\varphi:Y\to \mathbb{R}$ be a continuous function. 
If $Y$ is unbounded, we assume that  there exists $\epsilon>0$ such that
\begin{equation}\label{gr}
\varphi(z)\geq (1+\epsilon)\log|z|\ \text{for}\ |z|\gg1.
\end{equation}
We note that this framework includes the case when $Y$ 
is compact and $\varphi:Y\to \mathbb{R}$ is any continuous function. 
Following \cite[Appendix B]{SaffTotik} we introduce the 
\textit{weighted global extremal function},
\begin{equation}
V_{Y,\varphi}(z):=\sup\{u(z): u\in \mathcal{L}(\C^n), u\leq \varphi\ \text{on}\ Y\}\,,
\end{equation}
where $\mathcal{L}(\C^n)$ denotes the \textit{Lelong class} 
of plurisubharmonic (psh) functions $u$ that satisfies 
$$u(z)-\log^+|z|=O(1)$$ where $\log^+=\max(\log,0)$. 
We let 
$$\mathcal{L}^+(\mathbb{C}^m):=
\{u\in \mathcal{L}(\mathbb{C}^m): u(z)\geq \log^+|z|+C_u\ 
\text{for some}\ C_u\in \mathbb{R}\}.$$
In what follows, we let 
 $$g^*(z):=\limsup_{w\to z}g(w)$$
denote the upper semi-continuous regularization of $g$. 
Seminal results of Siciak and Zaharyuta (see \cite[Appendix B]{SaffTotik} 
and references therein) imply that 
$V^*_{Y,\varphi}\in\mathcal{L}^+(\mathbb{C}^m)$ and that  
$V_{\varphi}$ verifies
 \begin{equation}\label{envelope}
 V_{Y,\varphi}(z)=\sup\left\{\frac{1}{\deg f}\,\log|f(z)|:\,
 f\in\C[z]\,,\;\sup_{z\in Y}|f(z)|e^{-(\deg f)\varphi(z)}\leq 1\right\}.
 \end{equation}
 For $r>0$ let us denote 
 \[Y_r:=\{z\in Y:\,|z|\leq r\}\,.\] 
 Then it follows that $V_{Y,\varphi}=V_{Y_r,\varphi}$ 
 for sufficiently large $r\gg1$ (\cite[Appendix B, Lemma 2.2]{SaffTotik}). 
 
 A closed set $Y\subset \C^n$ is said to be \textit{locally regular at} $w\in Y$  if for every $\rho>0$ the extremal function $V_{Y\cap \overline{B(w,\rho)}}(z)$ is continuous at $w$. The set $Y$ is called locally regular if $Y$ is locally regular at each $w\in Y$. A classical result of Siciak \cite{Siciak} asserts that of $Y$ is locally regular and $\varphi$ is continuous weight function then the weighted extremal function $V_{Y,\varphi}$ is also continuous and hence $V_{Y,\varphi}=V_{Y,\varphi}^*$ on $\C^n$. In the rest of this section we assume that $Y$ is a locally regular closed set.

The psh function $V_{Y,\varphi}$ is locally bounded on $\C^n$ 
and hence by the Bedford-Taylor theory \cite{BT1,BT2} 
the \textit{weighted equilibrium measure} 
$$\mu_{Y,\varphi}:=\frac{1}{n!}(dd^cV_{Y,\varphi})^n$$ 
is well-defined and does not put any mass on pluripolar sets. 
Moreover, if $S_{Y,\varphi}:=\supp\mu_{Y,\varphi}$ then by 
\cite[Appendix B]{SaffTotik} we have 
$S_{Y,\varphi}\subset \{V_{Y,\varphi}=\varphi\}$. 
Hence, the support $S_{Y,\varphi}$ is a compact set. 
An important example is $Y=\C^n$ and 
$\varphi(z)=\frac{\|z\|^2}{2}$, which gives 
$\mu_{Y,\varphi}=\mathbbm{1}_{B}\,d\lambda_{2n}(z)$, 
where $\mathbbm{1}_{B}$ denotes the characteristic function 
of the Euclidean unit ball in $\C^n$.

A locally finite Borel measure $\nu$ is called a Berstein-Markov (BM) 
measure for the weighted set $(Y,\varphi)$ if $\nu$ is supported on $Y$,
\begin{equation}\label{gr2}
 \int_{Y\setminus Y_1}\frac{1}{|z|^a}\,d\nu(z)<\infty\ \text{for some}\ a>0\,,
\end{equation}
and $(Y_r,Q,\nu)$ satisfies the following weighted Bernstein-Markov inequality 
for all $r$ sufficiently large (see \cite[Section 6]{BL15}): 
there exist $M_p=M_p(r)\geq 1$ so that 
$\limsup_{p\to\infty} M_p^{1/p}=1$ and   
\begin{equation}\label{bm}
 \|fe^{-p\varphi}\|_{Y_r}:=
 \max_{z\in Y_r}|f(z)|e^{-p\varphi(z)}\leq 
 M_p \|fe^{-p\varphi}\|_{L^2(Y_r,\nu)}\,,\,\text{ for all }f\in\C_p[z].
\end{equation}
Conditions (\ref{gr}) and (\ref{gr2}) ensure that the measure $e^{-2n\varphi}d\nu$  
has finite moments of order up to $n$, while  condition (\ref{bm}) 
implies that asymptotically the $L^2(\nu)$ and $\sup$ norms 
of weighted polynomials are equivalent. We also remark that 
BM-measures always exist \cite{BMsurvey}. 

Let $d_p:=\dim\C_p[z]$.
We define an inner product on the space $\C_p[z]$ by 
\begin{equation}\label{n1}
(f,g)_{p}:=\int_Yf(z)\overline{g(z)}e^{-2p\varphi(z)}\,d\nu(z)\,.
\end{equation}
It gives the norm
\begin{equation}\label{n}
\|f\|_{p}^2:=\int_Y|f(z)|^2e^{-2p\varphi(z)}\,d\nu(z)\,.
\end{equation}
Let $\{P^p_j\}_{j=1}^{d_p}$ be a fixed orthonormal basis  for 
$\C_p[z]$ with respect to the inner product \eqref{n1}.

\subsection{Asymptotic zero distribution of random polynomials}
\label{random}
%We consider again the orthonormal basis $\{P_j^p\}_{j=1}^{d_p}$ 
%for $\C_p[z]$ from section \ref{SS:Bka}.  
A \textit{random polynomial} has the form 
\begin{equation}\label{rp}
f_p(z)=\sum_{j=1}^{d_p}a^p_jP_j^p(z)\,,
\end{equation}            
where $a^p_j$ are independent identically distributed (i.i.d) complex valued 
random variables. In this survey, we assume that distribution law of $a_j^p$ is of the form 
${\bf{P}}:=\phi(z) \,d\lambda_{2n}(z)$ where 
$\phi:\mathbb{C}\to [0,\infty)$ is a bounded function. 
Moreover, we assume that {\bf{P}} 
has sufficiently fast tail decay probabilities, 
that is for sufficiently  large $R>0$
\begin{equation}\label{ta}
{\bf{P}}\{a_j^p\in\mathbb{C}: |a_j^p|>e^R\}=O(R^{-\rho})\,,
\end{equation}
for some $\rho>n+1$. Recall that in the case 
of standard complex Gaussians the tail decay of the above integral is of 
order $e^{-R^2}$. We use the identification 
$\C_p[z]\simeq \mathbb{C}^{d_p}$ 
to endow the vector space $\C_p[z]$ with a $d_p$-fold product probability 
measure $\Prob_p$. We also consider the product probability space 
$\prod_{n=1}^{\infty}(\C_p[z],\Prob_p)$ whose elements are 
sequences of random polynomials. We remark that the probability space 
$(\C_p[z],\Prob_p)$ depends on the choice of ONB 
(i.\,e.\ the unitary identification 
$\C_p[z]\simeq \mathbb{C}^{d_p}$ given by (\ref{rp})) 
unless $\Prob_p$ is invariant under unitary transformations 
(eg.\ Gaussian ensemble).
 
\begin{Example}
For $Y=\mathbb{C},$ the weight $\varphi(z)=|z|^2/2$, 
and $\nu=\lambda_{2n}$,  $P^p_j(z)=\sqrt{\frac{p^{j+1}}{\pi j!}}\,z^j$ 
form an ONB for $\C_p[z]$ with respect to the norm $(\ref{n})$. 
Hence, a random polynomial is of the form
$$f_p(z)=\sum_{j=0}^{p}a_j^p\,\sqrt{\frac{p^{j+1}}{\pi j!}}\,z^j.$$ 
The scaled polynomials 
$$W_p(z)=\sum_{j=0}^pa_j^p\,\frac{1}{\sqrt{j!}}\,z^j$$ 
are known in the literature as Weyl polynomials. 
It follows from Theorem \ref{BL2} below that the zeros of $f_p$ 
are equidistributed with respect to the to normalized Lebesgue measure 
on the unit disk in the plane.
\end{Example}
For a random polynomial $f_p$ we denote its zero divisor in $\C^n$ by $Z_{f_p}$
and by $[Z_{f_p}]$ the current of integration along the divisor $Z_{f_p}$.
%$Z_{f_p}:=\{z\in\C^n:f_p(z)=0\}$. 
We also define the \textit{expected zero current} by
\begin{equation}
\langle\mathbb{E}[Z_{f_p}],\Phi\rangle:=
\int_{\C_p[z]}\langle [Z_{f_p}],\Phi\rangle\,d\Prob_p\,,
\end{equation} 
where $\Phi$ is a bidegree $(n-1,n-1)$ test form. 
One of the key results in \cite{BL15} (see also \cite{B6,B7}) is the following:

\begin{Theorem}\label{BL2}
Let $a_j^p$ be complex random variables verifying (\ref{ta}). Then
\begin{equation}\label{exdist}
\lim_{p\to \infty}\frac1p\,\mathbb{E}[Z_{f_p}]=dd^c V_{Y,\varphi}
\end{equation}
in the sense of currents. Moreover, almost surely
\begin{equation}\label{asdist}
\lim_{n\to \infty}\frac1p\,[Z_{f_p}]=dd^c V_{Y,\varphi}\,,
\end{equation} 
in the sense of currents.
\end{Theorem}
The key ingredient in the proof is a result about the asymptotic behavior
of the Bergman kernel of the space $(\C_p[z], (\cdot,\cdot)_p)$, see
\cite[Proposition 3.1 and (6.5)]{BL15}.

Recently, Theorem \ref{BL2} (\ref{asdist}) 
has been generalized by the first named author \cite{B10} 
to the setting of discrete random coefficients. 
More precisely, for a radial $\mathscr{C}^2$ weight function $\varphi$, 
assuming the random coefficients $a_j$ are non-degenerate 
it is proved that the moment condition
\begin{equation}
\mathbb{E}[(\log(1+|a_j|))^n]<\infty
\end{equation}
is necessary and suficifient for almost sure convergence \eqref{asdist} 
of the zero currents. More recently, in the unweighted setting 
(i.\,e.\ $\varphi\equiv 0$) in $\mathbb{C}^n$, 
Bloom-Dauvergne \cite{BloomD} obtained another generalization: 
for a regular compact set $K\subset \C^n$, 
zeros of random polynomials $f_p$ with i.i.d.\ coefficients 
verifying the tail assumption
\begin{equation}\label{ta2}
{\bf{P}}\{a_j\in\mathbb{C}: |a_j|>e^R\}= o(R^{-n})\,,
\quad R\to\infty,
\end{equation}
satisfy as $p\to\infty$,
$$\frac1p\,[Z_{f_p}]\to dd^cV_K \ \hbox{in probability}.$$

\subsection{Random polynomial mappings}\label{S:rpm} 
Next, we will give an extension of Theorem \ref{BL2} to the setting of 
random polynomial mappings. For $1\leq k\leq n$ we consider 
$k$-tuples $f_p:=(f_p^1,\dots,f_p^k)$ of random polynomials $f_p^j$ 
which are chosen independently at random from $(\C_p[z],\Prob_p)$. 
This gives a probability space $\big((\C_p[z])^k, \mathcal{F}_p\big)$. 
We also consider the product probability space 
$$\mathcal{H}_k:=\prod_{p=1}^{\infty}\big(\C_p[z]^k, \mathcal{F}_p\big). $$  
A \textit{random polynomial mapping} is of the form 
$$F_p:\C^n\to\C^k$$
$$F_p(z):=\big(f_p^1(z),\dots,f_p^k(z)\big)$$ 
where $f_p^j\in (\C_p[z],\Prob_p)$ are independent random polynomials.  
We also denote by $$\|F_p(z)\|^2:=\sum_{j=1}^k|f_p^j(z)|^2.$$
Note that since random coefficients $a_j^p$ have continuous distributions, by Bertini's Theorem the zero divisors $Z_{f_p^j}$ 
are smooth and intersect transversally for almost every system 
$(f_p^1,\dots,f_p^k)\in\C_p[z]^k$. 
Thus, the simultaneous zero locus 
\begin{eqnarray*}
Z_{f_p^1,\dots,f_p^k}&:=&\{z\in\C^n: \|F_p(z)\|=0\}\\
&=&\{z\in\C^n: f_p^1(z)=\dots=f_p^k(z)=0\}
\end{eqnarray*} is a complex submanifold of codimension $k$ 
and obtained as a complete intersection of individual zero loci. 
This implies that 
\begin{eqnarray*}
(dd^c\log\|F_p\|)^k &=&[Z_{f_p^1,\dots,f_p^k}]\\
&=& dd^c\log|f_p^1|\wedge\dots\wedge dd^c\log|f_p^k|
\end{eqnarray*}
 Then using an inductive argument (\cite[Corollary 3.3]{B6} 
 see also \cite{ShZ08, BloomS}) one can obtain ``probabilistic Lelong-Poincar\'e" 
 formula for the expected zero currents 
\begin{eqnarray}
\mathbb{E}[Z_{f_p^1,\dots,f_p^k}] &=
& \mathbb{E}[dd^c\log|f_p^1|\wedge\dots\wedge dd^c\log|f_p^k|]\\ \nonumber
&=& \mathbb{E}[Z_{f_p^1}]\wedge \dots \wedge \mathbb{E}[Z_{f_p^k}].
\end{eqnarray}
Then the following is an immediate corollary of  the uniform convergence 
of Bergman kernels to the weighted global extremal function 
\cite[Proposition 3.1 and (6.5)]{BL15}
together with a theorem of Bedford and Taylor \cite[\S 7]{BT2} 
on the convergence of Monge-Amp\`ere measures:
%----------
\begin{Corollary}
Under the hypotheses of Theorem \ref{BL2} and for each $1\leq k\leq n,$ we have
as $p\to\infty$,
$$\lim_{p\to \infty} \frac{1}{p^k}\,\mathbb{E}[Z_{f_p^1,\dots,f_p^k}]=
(dd^cV_{Y,\varphi})^k.$$
\end{Corollary}
%----------
Finally, we consider the asymptotic zero distribution of random polynomial mappings. 
The distribution of zeros of random polynomial mappings was studied in \cite{Sh08} 
in the unweighted case for Gaussian ensembles. 
The key result for the Gaussian ensembles is that the variances of linear statistics 
of zeros are summable and this gives almost sure convergence of random zero currents. 
More recently, the first named author \cite{B6,B7} 
considered the weighted case and for non-Gaussian random coefficients. 
In the latter setting, the variances are asymptotic to zero but they 
are no longer summable (see \cite[Lemma 5.2]{B6}). 
A different approach is developed in \cite{B6,B7} by using super-potentials 
(see \cite{DS11}) in order to obtain the asymptotic zero distribution 
of random polynomial mappings. 
The following result follows from the arguments in \cite[Theorem 1.2]{B6} 
(see also \cite[Theorem 1.2 ]{B7}): 
%----------
\begin{Theorem}[{\cite{B6,B7}}]
Under the hypotheses of Theorem \ref{BL2} and for each $1\leq k\leq n$ 
we have almost surely in $\mathcal{H}_k$ that 
$$\lim_{p\to\infty}\left[dd^c\left(\frac1p\log\|F_p\|\right)\right]^k=
(dd^cV_{Y,\varphi})^k\,,$$
in the sense of currents. 
 In particular, for any bounded domain $U\Subset\C^n$ we have 
 almost surely in $\mathcal{H}_n$ that 
 $$p^{-n}\#\{z\in U:f_p^1(z)=\dots=f_p^n(z)=0\}\to \mu_{Y,\varphi}(U)
 \,,\quad p\to\infty.$$
\end{Theorem}
%----------
\subsection{Central limit theorem for linear statistics}\label{S:clt} 
In this subsection, we consider the special case where $Y=\C^n$ and 
$\varphi:\C^n\to\mathbb{R}$ is a $\mathscr{C}^2$ 
weight function satisfying the growth condition \eqref{gr}. 
Throughout this part  we assume that the random coefficients 
$a_j^n$ are independent copies of standard complex normal distribution 
$N_{\mathbb{C}}(0,1)$. We consider the random variables
$$\mathcal{Z}_p^{\Phi}(f_p):= \langle [Z_{f_p}],\Phi\rangle $$ 
where $[Z_{f_p}]$ denotes the current of integration along the zero 
divisor of $f_p$ and $\Phi$ is a real $(n-1,n-1)$ test form. 

The random variables $\mathcal{Z}_{f_p}$ are often called 
\textit{linear statistics} of zeros. 
A form of universality for zeros of random polynomials is the
central limit theorem (CLT) for linear statistics of zeros, that is 
$$\frac{\mathcal{Z}_p^{\Phi}-\mathbb{E} 
\mathcal{Z}_p^{\Phi}}{\sqrt{\operatorname{Var}[\mathcal{Z}_p^{\Phi}]}}$$
converge in distribution to the (real) Gaussian random variable 
$\mathcal{N}(0,1)$ as $p\to \infty$. 
Here, $\mathbb{E}[\mathcal{Z}_p^{\Phi}]$ denotes the expected value and 
$\operatorname{Var}[\mathcal{Z}^{\Phi}_p]$ denotes
 the variance of the random variable $\mathcal{Z}^{\Phi}_p$.

\par In complex dimension one, Sodin and Tsirelson \cite{STr} proved the
asymptotic normality of $Z_p^{\psi}$ for Gaussian analytic functions 
and a $\mathscr{C}^2$ function $\psi$ by using the method of moments 
which is a classical approach in probability theory to prove CLT for Gaussian 
random processes whose variances are asymptotic to zero. 
More precisely, they observed that the asymptotic normality of linear statistics 
reduces to the Bergman kernel asymptotics (\cite[Theorem 2.2]{STr}) 
(see also \cite{NaSo2} for a generalization of this result to the case where 
$\psi$ is merely a bounded function by using a different method). 
On the other hand, Shiffman and Zelditch \cite{SZ3} pursued the idea 
of Sodin and Tsirelson and generalized their result to the setting of 
random holomorphic sections for a positive line bundle $L\to X$ 
defined over a projective manifold. Building upon ideas from \cite{STr,SZ3} 
and using the near and off diagonal Bergman kernel asymptotics 
(see \cite[\S2]{B4}) we proved a CLT for linear statistics:

\begin{Theorem}[{\cite[Theorem 1.2]{B4}}]\label{CLT2}
Let $\varphi:\C^n \to\mathbb{R}$ be a $\mathscr{C}^2$ 
weight function satisfying (\ref{gr}) and $\Phi$ be a real $(n-1, n-1)$ test form 
with $\mathscr{C}^3$ coefficients such that 
$\partial\overline{\partial}\Phi\not\equiv0$ and
$\partial\overline{\partial}\Phi$ is supported 
in the interior of the bulk of $\varphi$ (see \cite[(2.3)]{B4}).
Then the linear statistics
$$\frac{\mathcal{Z}_p^{\Phi}-
\mathbb{E} \mathcal{Z}_p^{\Phi}}{\sqrt{\operatorname{Var}[\mathcal{Z}_p^{\Phi}]}}$$ 
converge in distribution to the (real) Gaussian $\mathcal{N}(0,1)$ as $p\to \infty$.
\end{Theorem}

\subsection{Expected number of real zeros} In this part, we consider 
random univariate polynomials of the form
\[f_p(z)=\sum_{j=0}^pa^p_jc^p_jz^j\] 
where $c^p_j$ are deterministic constants 
and $a^p_j$ are real i.i.d.\ random variables of mean zero and variance one. 
We denote the number of real zeros of $f_p$ by $N_p.$ 
Therefore $N_p:\mathcal{P}_p\to \{0,\dots,p\}$ defines a random variable. 

The study of the number of real roots for Kac polynomials 
(i.\,e.\ $c^p_j=1$ and $a_j^p$ are i.i.d.\ real Gaussian) 
goes back to Bloch and P\'olya \cite{BPo} where they considered the 
case when the random variable $a_j^p$ is the uniform 
distribution on the set 
$\{-1,0,1\}$. This problem was also considered by 
Littlewood and Offord in a series of papers \cite{LO, LO1,LO2} 
for real Gaussian, Bernoulli and uniform distributions. 
In \cite{Kac} Kac established a remarkable formula for the expected 
number of real zeros of Gaussian random polynomials 
\begin{equation}\label{kac1}
\mathbb{E}[N_p]=\frac{4}{\pi}\int_0^1\frac{\sqrt{A(x)C(x)-B^2(x)}}{A(x)}\,dx\,,
\end{equation}
where $$A(x)=\sum_{j=0}^px^{2j},\ B(x)=
\sum_{j=1}^pjx^{2j-1},\ C(x)=\sum_{j=1}^pj^2x^{2j-2}\,,$$
and (\ref{kac1}) in turn implies that
\begin{equation}\label{kac2}
\mathbb{E}[N_p]=\big(\frac{2}{\pi}+o(1)\big)\log p.
\end{equation}
Later, Erd\H{o}s and Turan \cite{ET} obtained more precise estimates. 
The result stated in (\ref{kac2}) was also generalized to Bernoulli distributions 
by Erd\H{o}s and Offord \cite{EO} as well as to distributions in the domain of 
attraction of the normal law by Ibragimov and Maslova \cite{IM1,IM2}.

On the other hand, for models other than Kac ensembles, 
the behavior of $N_p$ changes considerably. In \cite{EdK},  
Edelman and Kostlan gave a beautiful geometric argument in order to 
calculate the expected number of real roots of Gaussian random polynomials. 
The argument in \cite{EdK} applies in the quite general setting of random sums of the form 
$$f_p(z)=\sum_{j=0}^pa_jP^p_j(z)\,,$$ where $P^p_j$ are entire functions 
that take real values on the real line (see also \cite{vanderbei, LPX} 
and references therein for a recent treatment of this problem). 
In particular, Edelman and Kostlan \cite[\S 3]{EdK} proved that 
 \[\mathbb{E}[N_p]=
 \begin{cases}  \sqrt{p} & 
 \text{for elliptic polynomials, i.e.\ when $c^p_j=
 \sqrt{\binom{p}{j}}$}\,, \\[6pt] 
\Big(\dfrac{2}{\pi}+o(1)\Big)\sqrt{p}  & \text{for Weyl polynomials, 
i.e.\ when $c^p_j=\sqrt{\frac{1}{j!}}\,\cdot$} 
\end{cases} \]

More recently, Tao and Vu \cite{TaoVu2} 
established some local universality results concerning the correlation
functions of the zeroes of random polynomials of the form (\ref{rp}). 
In particular, the results of \cite{TaoVu2} generalized the aforementioned ones 
for real Gaussians to the setting where $a_j^p$ is a random variable satisfying 
the moment condition $\mathbb{E}|\zeta|^{2+\delta}<\infty$ for some $\delta>0.$  

We remark that all the three models described above 
(Kac, elliptic and Weyl polynomials) arise in the context of orthogonal polynomials. 
In this direction, the expected distribution of real zeros for random linear combination 
of Legendre polynomials is studied by Das \cite{Das71} in which he proved that 
$$\mathbb{E}[N_p]=\frac{p}{\sqrt{3}}+o(p).$$ 

In terms of our model described in \S\ref{random} 
we have the following result on real zeros of random polynomials:

\begin{Theorem}[{\cite[Theorem 1.1]{B8}}]
Let $\varphi:\mathbb{C}\to \mathbb{R}$ be a non-negative 
radially symmetric weight function of class 
$\mathscr{C}^2$ satisfying (\ref{gr}) and $a_j^p$ be i.i.d.\ non-degenerate 
real random variables of mean zero and variance one satisfying 
$\mathbb{E}|a_j^p|^{2+\delta}<\infty$ for some $\delta>0.$ 
Then the expected number of real zeros of the random polynomials $f_{p}(z)$ satisfies  
\begin{equation}\label{azeros}
\lim_{n\to \infty}\frac{1}{\sqrt{p}}\,\mathbb{E}[N_p^{\zeta}]=
\frac{1}{\pi}\int_{S_{\varphi}\cap \mathbb{R}}\sqrt{\frac12\Delta\varphi(x)}\,dx\,.
\end{equation}
\end{Theorem}

%--------------------------------------------------------------------------------------

\section{Equidistribution for zeros of random holomorphic sections}\label{S:equidist}

We start by recalling a very general result about 
the equidistribution of zeros of random holomorphic sections 
of sequences of line bundles on an analytic space 
\cite[Theorem 1.1]{BCM}. Following \cite{CMM,BCM}, 
we consider the following setting:

\smallskip

(A1) $(X,\omega)$ is a compact (reduced) normal K\"ahler space 
of pure dimension $n$, $X_{\rm reg}$ denotes the set of 
regular points of $X$, and $X_{\rm sing}$ denotes the set 
of singular points of $X$.

\smallskip

(A2) $(L_p,h_p)$, $p\geq1$, is a sequence of holomorphic line 
bundles on $X$ with singular Hermitian metrics $h_p$ whose 
curvature currents verify 
\begin{equation}\label{e:pc}
c_1(L_p,h_p)\geq a_p\,\omega \, \text{ on $X$, where 
$a_p>0$ and } \lim_{p\to\infty}a_p=\infty.
\end{equation}
Let $A_p:=\int_Xc_1(L_p,h_p)\wedge\omega^{n-1}$. 
If $X_{\rm sing}\neq\emptyset$ we also assume that 
\begin{equation}\label{e:domin0}
\exists\,T_0\in\cT(X) \text{ such that } 
c_1(L_p,h_p)\leq A_pT_0\,,\;\forall\,p\geq1\,.
\end{equation}

\par Here $\cT(X)$ denotes the space of positive closed 
currents of bidegree $(1,1)$ on $X$ 
with local plurisubharmonic (psh) potentials. We refer to \cite[Section 2.1]{CMM} 
for the definition and main properties of psh 
functions and currents on analytic spaces, and to \cite[Section 2.2]{CMM} 
for the notion of singular Hermitian holomorphic line bundles on analytic spaces. 

We let $H^0_{(2)}(X,L_p)$ be the Bergman space 
of $L^2$-holomorphic sections of $L_p$ relative to the metric 
$h_p$ and the volume form $\omega^n/n!$ on $X$, 
\begin{equation}\label{e:bs}
H^0_{(2)}(X,L_p)=\left\{S\in H^0(X,L_p):\,
\|S\|_p^2:=\int_{X_{\rm reg}}|S|^2_{h_p}\,\frac{\omega^n}{n!}
<\infty\right\}\,,
\end{equation}
endowed with the obvious inner product.  
For $p\geq1$, let $d_p=\dim H^0_{(2)}(X,L_p)$ and let $S_1^p,\dots,S_{d_p}^p$ 
be an orthonormal 
basis of $H^0_{(2)}(X,L_p)$.  Note that the space $H^0(X,L)$ of 
holomorphic sections of a holomorphic line bundle $L$ on a compact 
analytic space $X$ is finite dimensional (see e.g.\ \cite[Th\'eor\`eme 1,\,p.27]{An:63}). 

\par By using the above orthonormal bases, we identify the spaces 
$\hp\simeq \C^{d_p}$ and we endow them with probability measures 
$\sigma_p$ verifying the moment condition: 

\smallskip

(B) There exist a constant $\nu\geq1$ and for every $p\geq1$ 
constants $C_p>0$ such that  
\[\int_{\C^{d_p}}\big|\log|\langle a,u\rangle|\,\big|^\nu\,d\sigma_p(a)\leq 
C_p\,,\,\text{ $\forall\,u\in\C^{d_p}$ with $\|u\|=1$}\,.\] 

\smallskip

Given a section $s\in H^0_{(2)}(X,L_p)$ we denote by $[s=0]$ 
the current of integration over the zero divisor of $s$.
We recall the Lelong-Poincar\'e formula (see e.g.\ \cite[Theorem 2.3.3]{MM07})
\begin{equation}\label{e:LP}
[s=0]=c_1(L_p,h_p)+dd^c\log|s|_{h_p}\,,
\end{equation}
where $d=\partial+\overline\partial$ and $d^c=
\frac{1}{2\pi i}(\partial-\overline\partial)$. 

The expectation current $\E[s_p=0]$ of the current-valued random variable 
$H^0_{(2)}(X,L_p)\ni s_p\mapsto[s_p=0]$ is defined by
$$\big\langle \E[s_p=0],\Phi\big\rangle=
\int\limits_{H^0_{(2)}(X,L_p)}\big\langle[s_p=0],\Phi\big\rangle\,d\sigma_p(s_p),$$
where $\Phi$ is a $(n-1,n-1)$ test form on $X$. 
We consider the product probability space
\[(\mathcal{H},\sigma)=
\left(\prod_{p=1}^\infty H^0_{(2)}(X,L_p),\prod_{p=1}^\infty\sigma_p\right). 
\]
The following result gives the distribution
of the  zeros of random
sequences of holomorphic sections of $L_p$, as well as the convergence
in $L^1$ of the logarithms of their pointwise norms. 
Note that the Lelong-Poincar\'e formula \eqref{e:LP} shows that the $L^1$ convergence of 
the logarithms of the pointwise norms of sections implies the weak convergence of their zero-currents. 

%%%%%%%%%%%%%%%%%
\begin{Theorem}[{\cite[Theorem 1.1]{BCM}}]\label{th1}
Assume that $(X,\omega)$, $(L_p,h_p)$ and $\sigma_p$ 
verify the assumptions (A1), (A2) and (B). 
Then the following hold:

\smallskip
\noindent
(i) If $\displaystyle\lim_{p\to\infty}C_pA_p^{-\nu}=0$ 
then $\displaystyle\frac{1}{A_p}\big(\E[s_p=0]-c_1(L_p,h_p)\big)\to 0$\,, 
as $p\to \infty$, in the weak sense of currents on $X$. 

\smallskip
\noindent
(ii) If $\displaystyle\liminf_{p\to\infty}C_pA_p^{-\nu}=0$ 
then there exists a sequence of natural numbers $p_j\nearrow\infty$ 
such that for $\sigma$-a.\,e.\ sequence $\{s_p\}\in\mathcal{H}$ we have
\begin{equation}\label{e:ss}
\frac{1}{A_{p_j}}\log|s_{p_j}|_{h_{p_j}}\to0\,,\,\;
\frac{1}{A_{p_j}}\big([s_{p_j}=0]-c_1(L_{p_j},h_{p_j})\big) \to0\,,\,
\text{ as $j\to\infty$,}
\end{equation}
in $L^1(X,\omega^n)$, respectively in the weak sense of currents on $X$.

\smallskip
\noindent
(iii) If $\displaystyle\sum_{p=1}^{\infty}C_pA_p^{-\nu}<\infty$ 
then for $\sigma$-a.\,e.\ sequence
$\{s_p\}\in\mathcal{H}$ we have 
\begin{equation}\label{e:cp}
\frac{1}{A_p}\log|s_p|_{h_p}\to0\,,\,\;
\frac{1}{A_p}\big([s_p=0]-c_1(L_p,h_p)\big) \to0\,,\,\text{ as $p\to\infty$,}
\end{equation}
in $L^1(X,\omega^n)$, respectively in the weak sense of currents on $X$.
\end{Theorem}
The key ingredient in the proof is a result about the asymptotic behavior
of the Bergman kernel of the space $H^0_{(2)}(X,L_p)$ defined in \eqref{e:bs}, see
\cite[Theorem 1.1]{CMM}.

%%%%%%%%%%%%%%%%%%%
\smallskip

Note that by \eqref{e:pc}, $A_p\geq a_p\int_X\omega^n$, 
hence $A_p\to\infty$ as $p\to\infty$. So if the measures $\sigma_p$ 
verify (B) with constants $C_p=\Gamma_\nu$ independent of $p$ then 
the hypothesis of $(i)$, $\lim_{p\to\infty}\Gamma_\nu A_p^{-\nu}=0$, 
is automatically verified. Moreover, the hypothesis of $(iii)$ becomes 
$\sum_{p=1}^{\infty}A_p^{-\nu}<\infty$.

\smallskip

General classes of measures $\sigma_p$ that satisfy condition (B) 
were given in \cite[Section 4.2]{BCM}. 
They include the Gaussians and the Fubini-Study volumes on 
${\mathbb C}^{d_p}$, which verify (B) for every $\nu\geq1$ 
with a constant $C_p=\Gamma_\nu$ independent of $p$. 
In Section \ref{S:mtrs}, we provide further examples of measures 
$\sigma_p$ that satisfy condition (B) and have support in totally 
real subsets of ${\mathbb C}^{d_p}$.

\medskip

We recall next several important special cases of Theorem \ref{th1}, 
as given in \cite{BCM}. Let $(L,h)$ be a fixed singular Hermitian holomorphic 
line bundle on $(X,\omega)$, and let $(L_p,h_p)=(L^p,h^p)$, 
where $L^p:=L^{\otimes p}$ and $h^p:=h^{\otimes p}$ 
is the singular Hermitian metric induced by $h$. In this case 
hypothesis \eqref{e:domin0} is automatically verified since 
$c_1(L^p,h^p)=p\,c_1(L,h)$. We have:

\begin{Corollary}[{\cite[Corollary 1.3]{BCM}}]\label{C:Lp1}
Let $(X,\omega)$ be a compact normal K\"ahler space and $(L,h)$ 
be a singular Hermitian holomorphic line bundle on $X$ such that 
$c_1(L,h)\geq\varepsilon\omega$ for some $\varepsilon>0$. 
For $p\geq1$ let $\sigma_p$ be probability measures on $H^0_{(2)}(X,L^p)$
satisfying condition (B).
Then the following hold:

\smallskip
\noindent
(i) If $\displaystyle\lim_{p\to\infty}C_p\,p^{-\nu}=0$ 
then $\displaystyle\frac{1}{p}\,\E[s_p=0]\to c_1(L,h)$\,, 
as $p\to \infty$, weakly on $X$. 

\smallskip
\noindent
(ii) If $\displaystyle\liminf_{p\to\infty}C_p\,p^{-\nu}=0$ 
then there exists a sequence of natural numbers $p_j\nearrow\infty$ 
such that for $\sigma$-a.\,e.\ sequence $\{s_p\}\in\mathcal{H}$ we have
as $j\to\infty$,
\[\frac{1}{p_j}\,\log|s_{p_j}|_{h^{p_j}}\to0
\,\;\text{in $L^1(X,\omega^n)$} \,,\:\:
\frac{1}{p_j}\,[s_{p_j}=0]\to c_1(L,h)\,,\,
\text{weakly on $X$}.\]

\smallskip
\noindent
(iii) If $\displaystyle\sum_{p=1}^{\infty}C_p\,p^{-\nu}<\infty$ 
then for $\sigma$-a.\,e.\ sequence
$\{s_p\}\in\mathcal{H}$ we have as $p\to\infty$,
\[\frac{1}{p}\,\log|s_p|_{h^p}\to0
\,\;\text{in $L^1(X,\omega^n)$} \,,\:\:
\frac{1}{p}\,[s_p=0] \to c_1(L,h)\,,\,
\text{weakly on $X$}.\]
\end{Corollary}

In a series of papers starting with \cite{ShZ99}, 
Shiffman and Zelditch consider the case when $(L,h)$ 
is a positive line bundle on a projective manifold $(X,\omega)$ 
and $\omega=c_1(L,h)$. One says in this case that $(X,\omega)$ 
is polarized by $(L,h)$ and since the Hermitian metric $h$ 
is smooth we have that $H^0_{(2)}(X,L^p)=H^0(X,L^p)$ 
is the space of global holomorphic sections of $L^p$. 
In \cite{ShZ99}, Shiffman and Zelditch were the first to study 
the asymptotic distribution of zeros of random sequences of 
holomorphic sections in this setting, in the case when the probability 
measure is the Gaussian or the normalized area measure on the unit sphere 
of $H^0(X,L^p)$. Their results were generalized later in the setting of 
projective manifolds with big line bundles endowed with singular Hermitian 
metrics whose curvature is a K\"ahler current in \cite{CM11,CM13,CM13b}, 
and to the setting of line bundles over compact normal K\"ahler spaces in 
\cite{CMM} and in Corollary \ref{C:Lp1}. Analogous equidistribution 
results for non-Gaussian ensembles are proved in \cite{DS06,B6,B7,BL15}.

%%%%%%%%%%%%%%%%%%%%

Theorem \ref{th1} allows one to handle the case of singular Hermitian 
holomorphic line bundles $(L,h)$ with positive curvature current 
$c_1(L,h)\geq0$ which is not a K\"ahler current (i.e.\ \eqref{e:pc} does not hold). 
Let $(X,\omega)$ be a K\"ahler manifold with a positive line bundle $(L,h_0)$, 
where $h_0$ is a smooth Hermitian metric $h_0$ such that $c_1(L,h_0)=\omega$. 
The set of singular Hermitian metrics $h$ on $L$ with $c_1(L,h)\geq0$ is in 
one-to-one correspondence to the set $PSH(X,\omega)$ of $\omega$-plurisubharmonic 
($\omega$-psh) functions on $X$, by associating to 
$\psi\in PSH(X,\omega)$ the metric $h_\psi=h_0e^{-2\psi}$ 
(see e.g., \cite{D90,GZ05}). 
Note that $c_1(L,h_\psi)=\omega+dd^c\psi$.
%------------
\begin{Corollary}[{\cite[Corollary 4.1]{BCM}}]\label{C:Lp3}
Let $(X,\omega)$ be a compact K\"ahler manifold and $(L,h_0)$ 
be a positive line bundle on $X$. Let $h$ be a singular Hermitian
metric on $L$ with $c_1(L,h)\geq0$ and let $\psi\in PSH(X,\omega)$
be its global weight associated by $h=h_0e^{-2\psi}$.
Let $\{n_p\}_{p\geq1}$ be a sequence of natural numbers such 
that 
\begin{equation}\label{n_p}
n_p\to\infty\:\:\text{and $n_p/p\to0$ as $p\to\infty$}.
\end{equation} 
Let  $h_p$ be the metric on $L^p$ given by
\begin{equation}\label{h_p}
h_p=h^{p-n_p}\otimes h_0^{n_p}.
\end{equation}
For $p\geq1$ let $\sigma_p$ be probability measures on 
$\hp=H^0_{(2)}(X,L^p,h_p)$
satisfying condition (B).
Then the following hold:

\smallskip
\noindent
(i) If $\displaystyle\lim_{p\to\infty}C_p\,p^{-\nu}=0$ 
then $\displaystyle\frac{1}{p}\,\E[s_p=0]\to c_1(L,h)$\,, 
as $p\to \infty$, weakly on $X$. 

\smallskip
\noindent
(ii) If $\displaystyle\liminf_{p\to\infty}C_p\,p^{-\nu}=0$ 
then there exists a sequence of natural numbers $p_j\nearrow\infty$ 
such that for $\sigma$-a.\,e.\ sequence $\{s_p\}\in\mathcal{H}$ we have
as $j\to\infty$,
\[\frac{1}{p_j}\log|s_{p_j}|_{h_{p_j}}\to\psi\;\;
\text{in $L^1(X,\omega^n)$}\,,\quad
\frac{1}{p_j}[s_{p_j}=0]\to c_1(L,h)\,,\,
\text{weakly on $X$.}\]

\smallskip
\noindent
(iii) If $\displaystyle\sum_{p=1}^{\infty}C_p\,p^{-\nu}<\infty$ 
then for $\sigma$-a.\,e.\ sequence
$\{s_p\}\in\mathcal{H}$ we have as $p\to\infty$,
\[\frac{1}{p}\log|s_p|_{h_p}\to\psi\;\;
\text{in $L^1(X,\omega^n)$}\,,\quad
\frac{1}{p}[s_p=0] \to c_1(L,h)\,,\,\text{weakly on $X$.}\]
\end{Corollary}

\medskip

In the case of polynomials in $\C^n$ the previous results take 
the following form (cf.\ \cite[Example 4.3]{BCM}). 
Consider $X={\mathbb P}^n$ and $L_p=\mO(p)$, $p\geq1$, 
where $\mathcal O(1)\to{\mathbb P}^n$ is the hyperplane line bundle.
Let ${\zeta\in\mathbb C}^n\mapsto[1:\zeta]\in{\mathbb P}^n$ 
be the standard embedding. 
It is well-known that the global holomorphic sections of $\mO(p)$ are given by homogeneous polynomials 
of degree $p$ in the homogeneous coordinates $z_0,\ldots,z_n$
on $\C^{n+1}$. On the chart 
$U_0=\{[1:\zeta]\in\mathbb{P}^n:\zeta\in\C^n\}\cong\C^n$ 
they coincide with the space $\C_p[\zeta]$ of polynomials of total degree at most $p$. 
Let $\omega_{\FS}$ denote the Fubini-Study 
K\"ahler form on $\mathbb{P}^n$ and $h_{\FS}$ be the Fubini-Study metric on $\mO(1)$,
so $c_1(\mO(1),h_{\FS})=\omega_{\FS}$\,. 
The set $PSH(\mathbb{P}^n,p\,\omega_{\FS})$
is in one-to-one correspondence 
to the set $p\mathcal{L}({\mathbb C}^n)$, where 
\[
\mathcal{L}({\mathbb C}^n):=\left\{\varphi\in PSH(\C^n):\,\exists\, C_\varphi\in\R \text{ such that
$\varphi(z)\leq \log^+\|z\|+C_\varphi$ on $\C^n$}\right\}
\]
is the Lelong class of entire psh functions with logarithmic  growth 
(cf.\ \cite[Section 2]{GZ05}). The $L^2$-space
\[
H^0_{(2)}(\mathbb{P}^n,\mO(p),h_p)=
\left\{s\in H^0(\mathbb{P}^n,\mO(p)):\int_{\mathbb{P}^n}|s|^2_{h_p}
\frac{\omega^n_{\FS}}{n!}<\infty\right\},
\]
is isometric to the $L^2$-space of polynomials
\begin{equation}\label{e:cp2}
\C_{p,(2)}[\zeta]=
\left\{f\in\C_p[\zeta] :\int_{\C^n}|f|^2e^{-2\varphi_p}
\frac{\omega^n_{\FS}}{n!}<\infty\right\}.
\end{equation}
If $\sigma_p$ are probability measures on $\C_{p,(2)}[\zeta]$
we denote by $\mathcal{H}$ the corresponding product probability space
\((\mathcal{H},\sigma)=
\left(\prod_{p=1}^\infty \C_{p,(2)}[\zeta],\prod_{p=1}^\infty\sigma_p\right) 
\).
%------------
\begin{Corollary}[{\cite[Corollary 4.4]{BCM}}]\label{C:pol2}
Consider a sequence of functions $\varphi_p\in p\mathcal{L}(\C^n)$
such that $dd^c\varphi_p\geq a_p\,\omega_{\FS}$ on $\C^n$, where
$a_p>0$ and $a_p\to\infty$ as $p\to\infty$. For $p\geq1$ let $\sigma_p$ 
be probability measures on $\C_{p,(2)}[\zeta]$ satisfying condition (B). 
Assume that $\sum_{p=1}^{\infty}C_pp^{-\nu}<\infty$\,.
Then for $\sigma$-a.\,e.\ sequence
$\{f_p\}\in\mathcal{H}$ we have as $p\to\infty$,
\begin{align*}
&\frac{1}{p}\Big(\log|f_p|-\varphi_p\Big)\to0\:\:
\text{in $L^1(\C^n,\omega_{\FS}^n)\,,\,$ hence in $L^1_{loc}(\C^n)$}\,,\\
&\frac{1}{p}\Big([f_p=0]-dd^c\varphi_p\Big)\to0\,,\:\:\text{weakly on $\C^n$}\,.
\end{align*}
\end{Corollary}
%------------
In particular, one has:
%------------
\begin{Corollary}[{\cite[Corollary 4.5]{BCM}}]\label{C:pol3}
Let $\varphi\in \mathcal{L}(\C^n)$
such that $dd^c\varphi\geq \varepsilon\,\omega_{\FS}$ on $\C^n$
for some constant $\varepsilon>0$. 
For $p\geq1$ construct the spaces $\C_{p,(2)}[\zeta]$ by setting
of $\varphi_p=p\varphi$ in \eqref{e:cp2} and let 
$\sigma_p$ be probability measures on $\C_{p,(2)}[\zeta]$
satisfying condition (B). If $\sum_{p=1}^{\infty}C_p\,p^{-\nu}<\infty$,
then for $\sigma$-a.\,e.\ sequence
$\{f_p\}\in\mathcal{H}$ we have as $p\to\infty$,
\begin{equation}\label{e:pp}
\frac{1}{p}\log|f_p|\to\varphi\:\:
\text{in $L^1(\C^n,\omega_{\FS}^n)$\,,}\quad 
\frac{1}{p}[f_p=0]\to dd^c\varphi\,,\:\:\text{weakly on $\C^n$\,.}
\end{equation}
\end{Corollary}
%------------
Corollary \ref{C:Lp3}  can also be applied to the setting of polynomials in $\C^n$
to obtain a version of Corollary \ref{C:pol3} for arbitrary $\varphi\in \mathcal{L}(\C^n)$.
%------------
\begin{Corollary}[{\cite[Corollary 4.6]{BCM}}]\label{C:pol4}
Let $\varphi\in \mathcal{L}(\C^n)$ and let $h$ be the singular Hermitian metric 
on $\mO(1)$ corresponding to $\varphi$. 
Let $\{n_p\}_{p\geq1}$ 
be a sequence of natural numbers such 
that \eqref{n_p} is satisfied.
Consider the metric $h_p$ on $\mO(p)$ given by
$h_p=h^{p-n_p}\otimes h_{\FS}^{n_p}$ (cf.\ \eqref{h_p}).
For $p\geq1$ let $\sigma_p$ be probability measures on 
$H^0_{(2)}(\mathbb{P}^n,\mO(p),h_p)\cong \C_{p,(2)}[\zeta]$ 
satisfying condition (B). 
If $\sum_{p=1}^{\infty}C_p\,p^{-\nu}<\infty$,
then for $\sigma$-a.\,e.\ sequence
$\{f_p\}\in\mathcal{H}$ we have \eqref{e:pp} as $p\to\infty$.
\end{Corollary}
%------------

%%%%%%%%%%%%%%%%%%%%%%%%%%%

\section{Measures with totally real support}\label{S:mtrs}

We give here several important examples of measures supported in 
${\mathbb R}^k\subset{\mathbb C}^k$, that satisfy condition (B). 
If  $a,v\in{\mathbb C}^k$ we set $a=(a_1,\ldots,a_k)$, $v=(v_1,\ldots,v_k)$, 
\[\langle a,v\rangle=\sum_{j=1}^ka_jv_j\,,\,\;|v|^2=\langle v,\overline v\rangle=\sum_{j=1}^k|v_j|^2\,.\]
Moreover, we let $\lambda_k$ be the Lebesgue measure on $\R^k\subset\C^k$.

\subsection{Preliminary lemma}\label{SS:pl}
Let us start with a technical lemma which will be needed to 
estimate the moments in (B). Let $\sigma$ be a probability 
measure supported in ${\mathbb R}^k\subset{\mathbb C}^k$. 
For $\nu\geq1$ and $v\in{\mathbb C}^k$ we define 
\[\begin{split}
I_k=I_k(\sigma,\nu)&=\int_{{\mathbb R}^k}\big|\log|a_1|\big|^\nu\,d\sigma(a)\,,\\
J_k(v)=J_k(v;\sigma,\nu)&=
\int_{{\mathbb R}^k}\big|\log|\langle a,v\rangle|\big|^\nu\,d\sigma(a)\,,\\
K_k(v)=K_k(v;\sigma,\nu)&=
\int_{\{a\in{\mathbb R}^k:
\,|a_1||v|>1/\sqrt2\}}\big|\log(|a_1||v|)\big|^\nu\,d\sigma(a)\,.
\end{split}\]
Note that if $u\in{\mathbb C}^k$, $|u|=1$, $J_k(u)$ 
is the logarithmic moment of $\sigma$ considered in (B).

\begin{Lemma}\label{L:ss}
Let $\sigma$ be a rotation invariant probability measure supported in 
${\mathbb R}^k\subset{\mathbb C}^k$ and let $\nu\geq1$. 
If $u=s+it\in{\mathbb C}^k$, $s,t\in{\mathbb R}^k$, $|u|=1$ 
and $|s|\geq|t|$ then 
\[J_k(u)\leq 2^{2\nu-1}I_k+2^{\nu-1}K_k(t)+2^\nu\,.\]
Moreover if $\sigma$ is supported in the closed unit ball 
$B_k\subset{\mathbb R}^k$ then 
\[J_k(u)\leq 2^{\nu-1}I_k+1\,.\]
\end{Lemma}

\begin{proof}
Let 
\[\begin{split}
&E_1=\{a\in{\mathbb R}^k:\,|\langle a,u\rangle|\leq1\}\,,\,\;
E_2=\{a\in{\mathbb R}^k:\,|\langle a,u\rangle|>1\}\,,\\
&J_k(u)=J_k^1(u)+J_k^2(u)\,,\,\;J_k^j(u):=
\int_{E_j}\big|\log|\langle a,u\rangle|\big|^\nu\,d\sigma(a)\,.
\end{split}\]

Note that $|\langle a,u\rangle|^2=\langle a,s\rangle^2+
\langle a,t\rangle^2$ and $|s|\geq1/\sqrt2$. On $E_1$ we have 
$|\langle a,s\rangle|\leq|\langle a,u\rangle|\leq1$, hence 
\begin{equation}\label{e:I1}
J_k^1(u)\leq\int_{E_1}\big|\log|\langle a,s\rangle|\big|^\nu\,d\sigma(a)\leq J_k(s)\,,
\end{equation}
where 
\[J_k(s)=\int_{{\mathbb R}^k}\big|\log|\langle a,s\rangle|\big|^\nu\,d\sigma(a)
=\int_{{\mathbb R}^k}\big|\log(|a_1||s|)\big|^\nu\,d\sigma(a)\,,\]
since $\sigma$ is rotation invariant. Using Jensen's inequality and 
$1\geq|s|\geq1/\sqrt2\,$ we get 
\[\big|\log(|a_1||s|)\big|^\nu\leq\left(\big|\log|a_1|\big|+
\log\sqrt2\right)^\nu\leq2^{\nu-1}\big|\log|a_1|\big|^\nu+\frac{(\log 2)^\nu}{2}\,.\]
Therefore
\begin{equation}\label{e:I2}
J_k(s)\leq\int_{{\mathbb R}^k}\left(2^{\nu-1}\big|\log|a_1|\big|^\nu+
\frac{1}{2}\right)\,d\sigma(a)=2^{\nu-1}I_k+\frac{1}{2}\,.
\end{equation}
If $\supp\sigma\subseteq B_k$ then, since $B_k\subset E_1$, we have
\[J_k(u)=\int_{B_k}\big|\log|\langle a,u\rangle|\big|^\nu\,d\sigma(a)=J_k^1(u)\leq J_k(s)\,,\]
so the lemma follows from \eqref{e:I2}. 

We estimate next $J_k^2(u)$. To this end we write $E_2=E_2^+\cup E_2^-$, where 
\[E_2^+:=E_2\cap\{a\in{\mathbb R}^k:\,
|\langle a,s\rangle|\geq|\langle a,t\rangle|\}\,,\,\;E_2^-:=
E_2\cap\{a\in{\mathbb R}^k:\,|\langle a,s\rangle|<|\langle a,t\rangle|\}\,.\]
On $E_2^+$ we have  
$1<|\langle a,u\rangle|^2\leq2\langle a,s\rangle^2$, so 
$0<\log|\langle a,u\rangle|\leq\big|\log|\langle a,s\rangle|\big|+\log\sqrt2$. 
Jensen's inequality yields that 
\[\big|\log|\langle a,u\rangle|\big|^\nu\leq
2^{\nu-1}\big|\log|\langle a,s\rangle|\big|^\nu+\frac{(\log 2)^\nu}{2}\,.\]
We obtain
\[\int_{E_2^+}\big|\log|\langle a,u\rangle|\big|^\nu\,d\sigma(a)\leq
\int_{E_2^+}\left(2^{\nu-1}\big|\log|\langle a,s\rangle|\big|^\nu+
\frac{1}{2}\right)\,d\sigma(a)\leq2^{\nu-1}J_k(s)+\frac{1}{2}\,,\]
and by \eqref{e:I2},
\begin{equation}\label{e:I3}
\int_{E_2^+}\big|\log|\langle a,u\rangle|\big|^\nu\,d\sigma(a)\leq
2^{2(\nu-1)}I_k+2^{\nu-2}+\frac{1}{2}\,.
\end{equation}

Finally, on $E_2^-$ we have  $1<|\langle a,u\rangle|^2\leq
2\langle a,t\rangle^2$, so 
$E_2^-\subseteq \{a\in{\mathbb R}^k:\,|\langle a,t\rangle|>1/\sqrt2\}$. 
Hence by Jensen's inequality, 
\[\big|\log|\langle a,u\rangle|\big|^\nu\leq
2^{\nu-1}\big|\log|\langle a,t\rangle|\big|^\nu+\frac{(\log 2)^\nu}{2}\,.\]
Therefore
\[\int_{E_2^-}\big|\log|\langle a,u\rangle|\big|^\nu\,d\sigma(a)\leq
2^{\nu-1}\int_{\{a\in{\mathbb R}^k:\,|\langle a,t\rangle|>1/\sqrt2\}}
\big|\log|\langle a,t\rangle|\big|^\nu\,d\sigma(a)+\frac{1}{2}\,.\]
Since $\sigma$ is rotation invariant this implies that 
\begin{equation}\label{e:I4}
\int_{E_2^-}\big|\log|\langle a,u\rangle|\big|^\nu\,d\sigma(a)\leq
2^{\nu-1}K_k(t)+\frac{1}{2}\,.
\end{equation}

By \eqref{e:I3} and \eqref{e:I4} we get 
\begin{equation}\label{e:I5}
J_k^2(u)\leq 2^{2(\nu-1)}I_k+2^{\nu-1}K_k(t)+2^{\nu-2}+1\,.
\end{equation}
Hence by \eqref{e:I1}, \eqref{e:I2} and \eqref{e:I5},
\[J_k(u)\leq\left(2^{\nu-1}+2^{2(\nu-1)}\right)I_k+
2^{\nu-1}K_k(t)+2^{\nu-2}+\frac{3}{2}\leq2^{2\nu-1}I_k+2^{\nu-1}K_k(t)+2^\nu\,.\]
\end{proof}

\subsection{Real Gaussians}\label{SS:rG}
Let  $\sigma_k$ be the measure on $\R^k\subset\C^k$ with density 
\begin{equation}\label{e:Gauss}
d\sigma_k(a)=\pi^{-k/2}\,e^{-|a|^2}\,d\lambda_k(a)\,,
\end{equation}
where $a=(a_1,\ldots,a_k)\in\R^k$.

\begin{Proposition}\label{P:Gauss} For every integer $k\geq1$, 
every $\nu\geq1$, and every $u\in\C^k$ with $|u|=1$, we have 
\[J_k(u;\sigma_k,\nu)=\int_{\R^k}\big|\log|\langle a,u\rangle|\big|^\nu\,d\sigma_k(a)
\leq\Gamma_\nu:=2^{2\nu}\int_0^{+\infty}\big|\log x\big|^\nu e^{-x^2}\,dx+2^\nu\,.\]
\end{Proposition}

\begin{proof}
We write $u=s+it$, $s,t\in\R^k$, and assume without loss of generality 
that $|s|\geq|t|$. Fubini's theorem implies that 
\[\begin{split}
I_k&=\int_{{\mathbb R}^k}\big|\log|a_1|\big|^\nu\,d\sigma_k(a)=
2\pi^{-1/2}\int_0^{+\infty}\big|\log x\big|^\nu e^{-x^2}\,dx\,,\\
K_k(t)&=2\pi^{-1/2}\int_{\{x|t|>1/\sqrt2\}}\big|\log(x|t|)\big|^\nu 
e^{-x^2}\,dx=\frac{2}{|t|\sqrt\pi}
\int_{1/\sqrt2}^{+\infty}\big|\log x\big|^\nu e^{-x^2/|t|^2}\,dx\,.
\end{split}\]
Note that $1=|u|^2\geq2|t|^2$. For $x\geq1/\sqrt2\,$ we have 
\[e^{-x^2/|t|^2}=e^{-x^2/(2|t|^2)}e^{-x^2/(2|t|^2)}\leq 
e^{-x^2}e^{-1/(4|t|^2)}\,,\]
so 
\[K_k(t)\leq\frac{2}{\sqrt\pi}\;|t|^{-1}e^{-1/(4|t|^2)}\int_0^{+\infty}
\big|\log x\big|^\nu e^{-x^2}\,dx\,.\]
Note that $f(y):=ye^{-y^2/4}\leq f(\sqrt2)=\sqrt{2/e}<1$ for $y\geq0$. 
Thus 
\[K_k(t)<2\pi^{-1/2}\int_0^{+\infty}\big|\log x\big|^\nu e^{-x^2}\,dx\,.\]
Lemma \ref{L:ss} together with the above estimates on $I_k$ and $K_k(t)$ implies that
\[J_k(u)\leq2\pi^{-1/2}\big(2^{2\nu-1}+
2^{\nu-1}\big)\int_0^{+\infty}\big|\log x\big|^\nu e^{-x^2}\,dx+2^\nu
\leq2^{2\nu}\int_0^{+\infty}\big|\log x\big|^\nu e^{-x^2}\,dx+2^\nu\,.\]
\end{proof}

\subsection{Radial densities}\label{SS:rd}
Let  $\sigma_k$ be the measure on $\R^k\subset\C^k$ with density 
\begin{equation}\label{e:rd}
d\sigma_k(a)=\frac{\Gamma(\frac{k}{2}+\alpha)}{\Gamma(\alpha)
\pi^{\frac{k}{2}}}\,(1+|a|^2)^{-\frac{k}{2}-\alpha}\,d\lambda_k(a)\,,
\end{equation}
where $\alpha>0$ is fixed, $a=(a_1,\ldots,a_k)\in\R^k$, and $\Gamma$ 
is the Euler Gamma function.

\begin{Proposition}\label{P:rd} The measure $\sigma_k$ is a probability measure. 
For every integer $k\geq1$, every $\nu\geq1$, and every $u\in\C^k$ with $|u|=1$, 
we have 
\[J_k(u;\sigma_k,\nu)=
\int_{\R^k}\big|\log|\langle a,u\rangle|\big|^\nu\,d\sigma_k(a)
\leq\Gamma_{\alpha,\nu}\,,\]
where $\Gamma_{\alpha,\nu}>0$ 
is a constant depending only on $\alpha$ and $\nu$.
\end{Proposition}

\begin{proof} Recall that the area of the unit sphere ${\mathbf S}^{k-1}\subset\R^k$ is 
\begin{equation}\label{e:areaSk}
s_k=\frac{2\pi^{\frac{k}{2}}}{\Gamma(\frac{k}{2})}\,.
\end{equation}
Using spherical coordinates on $\R^k$ and changing variables $r^2=x/(1-x)$, we obtain 
\[\begin{split}
\int_{\R^k}\frac{d\lambda_k(a)}{(1+|a|^2)^{\frac{k}{2}+\alpha}}&=
s_k\int_0^{+\infty}\frac{r^{k-1}\,dr}{(1+r^2)^{\frac{k}{2}+\alpha}}=
\frac{s_k}{2}\int_0^1x^{\frac{k}{2}-1}(1-x)^{\alpha-1}\,dx \\
&=\frac{\pi^{\frac{k}{2}}}{\Gamma(\frac{k}{2})}\,B(k/2,\alpha)=
\frac{\pi^{k/2}\Gamma(\alpha)}{\Gamma(\frac{k}{2}+\alpha)}\,.
\end{split}\]
Hence $\int_{\R^k}d\sigma_k(a)=1$. 

\smallskip

Let us consider the following function which will be useful 
in estimating the integrals $I_k$ and $K_k$:
\[F_{k,\alpha}(y):=
s_{k-1}\int_0^{+\infty}\frac{r^{k-2}\,dr}{(1+y^2+r^2)^{\frac{k}{2}+\alpha}}
\,,\,\;y\in\R\,.\]
Changing variables $r^2=(1+y^2)\frac{x}{1-x}$ we get 
\[F_{k,\alpha}(y)=
\frac{s_{k-1}}{2}\,B\left(\frac{k-1}{2},\alpha+
\frac{1}{2}\right)(1+y^2)^{-\alpha-\frac{1}{2}}=
\frac{\pi^{\frac{k-1}{2}}\Gamma(\alpha+\frac{1}{2})}
{\Gamma(\frac{k}{2}+\alpha)(1+y^2)^{\alpha+\frac{1}{2}}}\,.\]
We next evaluate $I_k$ by using spherical coordinates 
for $(a_2,\ldots,a_k)\in\R^{k-1}$. We have
\[\begin{split}
I_k&=\int_{\R^k}\big|\log|a_1|\big|^\nu\,d\sigma_k(a)=
\frac{\Gamma(\frac{k}{2}+\alpha)}{\Gamma(\alpha)\pi^{\frac{k}{2}}}\,
\int_{-\infty}^{+\infty}\int_0^{+\infty}
\frac{s_{k-1}r^{k-2}\,\big|\log|y|\big|^\nu}
{(1+y^2+r^2)^{\frac{k}{2}+\alpha}}\,\;dr\,dy\\
&=2\,\frac{\Gamma(\frac{k}{2}+\alpha)}
{\Gamma(\alpha)\pi^{\frac{k}{2}}}\,\int_0^{+\infty}\big|\log|y|\big|^\nu 
F_{k,\alpha}(y)\,dy=\frac{2\Gamma(\alpha+\frac{1}{2})}
{\sqrt\pi\,\Gamma(\alpha)}\,\int_0^{+\infty}\frac{\big|\log|y|\big|^\nu}
{(1+y^2)^{\alpha+\frac{1}{2}}}\,dy=:C_{\alpha,\nu}\,.
\end{split}\]
We now write $u=s+it$, $s,t\in\R^k$, 
and assume without loss of generality that $|s|\geq|t|$. 
We estimate $K_k$ by the same method used for $I_k$:
\[\begin{split}
K_k(t)&=\int_{\{a\in{\mathbb R}^k:\,|a_1||t|>1/\sqrt2\}}
\big|\log(|a_1||t|)\big|^\nu\,d\sigma_k(a)\\
&=\frac{\Gamma(\frac{k}{2}+\alpha)}
{\Gamma(\alpha)\pi^{\frac{k}{2}}}\,\int_{\{y\in\R:\,|y||t|>1/\sqrt2\}}
\int_0^{+\infty}\frac{s_{k-1}r^{k-2}\,\big|\log(|y||t|)\big|^\nu}
{(1+y^2+r^2)^{\frac{k}{2}+\alpha}}\,\;dr\,dy\\
&=2\,\frac{\Gamma(\frac{k}{2}+\alpha)}{\Gamma(\alpha)\pi^{\frac{k}{2}}}\,
\int_{\frac{1}{|t|\sqrt2}}^{+\infty}\big|\log(y|t|)\big|^\nu F_{k,\alpha}(y)\,dy=
\frac{2\Gamma(\alpha+\frac{1}{2})}{\sqrt\pi\,\Gamma(\alpha)}\,
\int_{\frac{1}{|t|\sqrt2}}^{+\infty}
\frac{\big|\log(y|t|)\big|^\nu}{(1+y^2)^{\alpha+\frac{1}{2}}}\,dy\,.
\end{split}\]
Since $|t|\leq1$, making the substitution $x=y|t|$ yields
\[K_k(t)=\frac{2\Gamma(\alpha+\frac{1}{2})}
{\sqrt\pi\,\Gamma(\alpha)}\,|t|^{2\alpha}
\int_{\frac{1}{\sqrt2}}^{+\infty}
\frac{\big|\log x\big|^\nu}{(|t|^2+x^2)^{\alpha+\frac{1}{2}}}\,dx
\leq
\frac{2\Gamma(\alpha+\frac{1}{2})}{\sqrt\pi\,\Gamma(\alpha)}\,
\int_{\frac{1}{\sqrt2}}^{+\infty}\frac{\big|\log x\big|^\nu}{x^{2\alpha+1}}\,dx
=:C'_{\alpha,\nu}\,.\]
Lemma \ref{L:ss} together with the above estimates on $I_k$ and $K_k(t)$ 
implies that
\[J_k(u)\leq2^{2\nu-1}C_{\alpha,\nu}+2^{\nu-1}C'_{\alpha,\nu}+
2^\nu=:\Gamma_{\alpha,\nu}\,,\]
which concludes the proof.
\end{proof}

%------------------------------------------------------
\subsection{Area measure of real spheres}\label{SS:rs}
Let ${\mathcal A}_k$ be the surface measure on the unit sphere 
${\mathbf S}^{k-1}\subset\R^k\subset{\mathbb C}^k$, 
so ${\mathcal A}_k\big({\mathbf S}^{k-1}\big)=s_k$, where $s_k$ is given in \eqref{e:areaSk}, and let

\begin{equation}\label{e:sphere}
\sigma_k=\frac{1}{s_k}\,{\mathcal A}_k\,.
\end{equation}

\begin{Proposition}\label{P:sphere} For every integer $k\geq2$, every $\nu\geq1$, and every $u\in\C^k$ with $|u|=1$, we have 
\[J_k(u;\sigma_k,\nu)=\int_{{\mathbf S}^{k-1}}\big|\log|\langle a,u\rangle|\big|^\nu\,d\sigma_k(a)
\leq M_\nu\,(\log k)^\nu\,,\]
where $M_\nu>0$ is a constant depending only on $\nu$.
\end{Proposition}

\begin{proof}  We let $k\geq3$ and consider spherical coordinates on ${\mathbf S}^{k-1}$ such that 
\begin{align*}
&(\theta_1,\ldots,\theta_{k-2},\varphi)\in\left[-\frac{\pi}{2}\,,\frac{\pi}{2}\right]^{k-2}\times[0,2\pi]\,,\\
&a_1=\cos\varphi\,\cos\theta_1\,\ldots\,\cos\theta_{k-2}\,,\\
&a_2=\sin\varphi\,\cos\theta_1\,\ldots\,\cos\theta_{k-2}\,,\\
&a_3=\sin\theta_1\,\cos\theta_2\,\ldots\,\cos\theta_{k-2}\,,\\
& \hspace{17mm} \ldots\\
&a_k=\sin\theta_{k-2}\,,\\
&d{\mathcal A}_k=\cos\theta_1\,\cos^2\theta_2\,\ldots\,\cos^{k-2}\theta_{k-2}\;
d\theta_1\ldots d\theta_{k-2}d\varphi\,.
\end{align*}

Lemma \ref{L:ss} implies that 
\[J_k(u)\leq2^{\nu-1}\int_{{\mathbf S}^{k-1}}\big|\log|a_k|\big|^\nu\,d\sigma_k(a)+1\,.\]
Using spherical coordinates we obtain
\[I_k:=\int_{{\mathbf S}^{k-1}}\big|\log|a_k|\big|^\nu\,d\sigma_k(a)=\frac{\int_{{\mathbf S}^{k-1}}\big|\log|a_k|\big|^\nu\,d{\mathcal A}_k(a)}{\int_{{\mathbf S}^{k-1}}\,d{\mathcal A}_k(a)}=\frac{I'_k}{C_{k-2}}\,,\]
where 
\[I'_k:=\int_0^{\pi/2}\big|\log\sin t\big|^\nu\cos^{k-2}t\,dt=\int_0^1\big(-\log x\big)^\nu(1-x^2)^{\frac{k-3}{2}}\,dx\,,\,\;C_k:=\int_0^{\pi/2}\cos^kt\,dt\,.\]

One has that there exist constants $A,B>0$ such that $A\leq C_k\sqrt{k}\leq B$ for all $k\geq1$ (see e.g. the proof of \cite[Lemma 4.3]{CMM}). Lemma \ref{L:int} applied with $\tau=1/k$ and $b=\frac{k-3}{2}$ yields that
\[I'_k\leq \frac{3(\log k)^\nu+2^{\nu+1}\nu^\nu e^{-\nu}}{\sqrt k}\,.\]
Hence $I_k\leq A^{-1}\left(3(\log k)^\nu+2^{\nu+1}\nu^\nu e^{-\nu}\right)$, and the proof is complete. 
\end{proof}

\begin{Lemma}\label{L:int}
If $\nu\geq1$, $b\geq0$, and $0<\tau<1$, then
\[\int_0^1\big(-\log x\big)^\nu(1-x^2)^b\,dx\leq2^{\nu+1}\left(\frac{\nu}{e}\right)^\nu\sqrt\tau\,+2\big(-\log\tau\big)^\nu\left(b+\frac{3}{2}\right)^{-\frac{1}{2}}\,.\]
\end{Lemma}

\begin{proof} We have 
\[\int_0^1\big(-\log x\big)^\nu(1-x^2)^b\,dx\leq\int_0^\tau\big(-\log x\big)^\nu\,dx+\big(-\log\tau\big)^\nu\int_\tau^1(1-x^2)^b\,dx\,.\]
Since $f(x):=\big(-\log x\big)^\nu\sqrt x\leq f(e^{-2\nu})=(2\nu/e)^\nu$ for $0<x\leq1$, we obtain
\[\int_0^\tau\big(-\log x\big)^\nu\,dx\leq\left(\frac{2\nu}{e}\right)^\nu\int_0^\tau x^{-\frac{1}{2}}\,dx=2^{\nu+1}\left(\frac{\nu}{e}\right)^\nu\sqrt\tau\,.\]
Moreover
\[\int_\tau^1(1-x^2)^b\,dx\leq\int_0^1(1-x^2)^b\,dx=\frac{1}{2}\,\int_0^1\frac{(1-y)^b}{\sqrt y}\,dy=\frac{1}{2}\,B\left(\frac{1}{2}\,,b+1\right)=\frac{\sqrt\pi\,\Gamma(b+1)}{2\Gamma\left(b+\frac{3}{2}\right)}\,.\]
Stirling's formula
\[\Gamma(x)=\sqrt{\frac{2\pi}{x}}\,\left(\frac{x}{e}\right)^xe^{\mu(x)}\,,\,\;0<\mu(x)<\frac{1}{12x}\,,\]
where $x>0$, yields that 
\[\frac{\sqrt\pi\,\Gamma(b+1)}{2\Gamma\left(b+\frac{3}{2}\right)}\leq\frac{\sqrt{\pi e}}{2}\,
\left(\frac{b+\frac{3}{2}}{b+1}\right)^{\frac{1}{2}}\left(\frac{b+1}{b+\frac{3}{2}}\right)^{b+1}\,\left(b+\frac{3}{2}\right)^{-\frac{1}{2}}\,e^{\frac{1}{12(b+1)}}\,.\]
Since the function $y\mapsto\big(y(y+\frac{1}{2})^{-1}\big)^y$ is decreasing for $y>0$, it follows that 
\[\int_\tau^1(1-x^2)^b\,dx\leq\sqrt{\frac{\pi e}{6}}\;e^{\frac{1}{12}}\left(b+\frac{3}{2}\right)^{-\frac{1}{2}}<2\left(b+\frac{3}{2}\right)^{-\frac{1}{2}}\,.\]
The lemma now follows.
\end{proof}

Note that, by Proposition \ref{P:sphere}, the hypotheses of Theorem \ref{th1} for the measures $\sigma_{d_p}$ in \eqref{e:sphere} involve the quantity $\displaystyle\frac{\log d_p}{A_p}$. General examples of sequences of line bundles $L_p$ for which $\displaystyle\lim_{p\to\infty}\frac{\log d_p}{A_p}=0$ are given in \cite[Proposition 4.4]{CMM} and \cite[Proposition 4.5]{CMM} (see also \cite[Section 4.2.3]{BCM}).

%------------------------------------------------------
\subsection{Random holomorphic sections with i.i.d.\ real coefficients}\label{SS:riid}
Let $\sigma_k$ be the measure on $\R^k\subset\C^k$ with density 
\begin{equation}\label{e:riid}
d\sigma_k(a)=\phi(a_1)\ldots\phi(a_k)\,d\lambda_k(a)\,,
\end{equation}
where $a=(a_1,\ldots,a_k)\in\R^k$ and $\phi:\R\to[0,M]$ is a function such that $\displaystyle\int_{-\infty}^{+\infty}\phi(x)\,dx=1$. Note that this measure is in general not rotation invariant. 

Endowing the Bergman space $H^0_{(2)}(X,L_p)$ of holomorphic sections with the measures $\sigma_{d_p}$ from \eqref{e:riid} means the following: if $(S_j^p)_{j=1}^{d_p}$ is a fixed orthonormal basis of $H^0_{(2)}(X,L_p)$, then random holomorphic sections are of the form 
\[s_p=\sum_{j=1}^{d_p}a_j^pS_j^p\,,\]
where the coefficients $\{a_j^p\}_{j=1}^{d_p}$ are i.i.d.\ real random variables with density function $\phi$.

\begin{Proposition}\label{P:riid} Let $\sigma_k$ be the measure on $\R^k\subset\C^k$ defined in \eqref{e:riid} and assume that there exist $c>0$, $\rho>1$, such that  
\begin{equation}\label{tailc} 
\int_{\big\{x\in\R:\,|x|>e^R\big\}}\phi(x)\,dx\leq cR^{-\rho} \,,\,\;\forall\,R>0.
\end{equation}
For every integer $k\geq1$, every $1\leq\nu<\rho$, and every $u\in\C^k$ with $|u|=1$, we have 
\[J_k(u;\sigma_k,\nu)=\int_{\R^k}\big|\log|\langle a,u\rangle|\big|^\nu\,d\sigma_k(a)
\leq \Gamma k^{\nu/\rho}\,,\]
where $\Gamma=\Gamma(M,c,\rho,\nu)>0$.
\end{Proposition}

\begin{proof} We follow the proof of \cite[Lemma 4.13]{BCM}, making the necessary modifications since $\sigma_k$ is supported on $\R^k$ and $u\in\C^k$. Since  
\[\left\{a\in\R^k:\,|\langle a,u\rangle|>e^R\right\}\subset \bigcup_{j=1}^k\big\{a_j\in\R:\,|a_j|>e^{R-\frac12\log k}\big\},\]
for $R\geq\log k$ we have by (\ref{tailc}),
\begin{equation}\label{ues}
\sigma_k\big(\{a\in\R^k:\,\log|\langle a,u\rangle|>R\}\big)\leq 
k\,\int_{\big\{a_j\in\R:\,|a_j|>e^{R-\frac12\log k}\big\}}\phi(x)\,dx\leq 
\frac{2^\rho ck}{R^\rho}\,\cdot
\end{equation}

We write $u=s+it$, where $s,t\in\R^k$. Since $1=|u|^2=|s|^2+|t|^2$ we may assume that $|s_1|\geq(2k)^{-1/2}$, where $s=(s_1,\ldots,s_k)$. As $|\langle a,s\rangle|\leq|\langle a,u\rangle|$ and $\phi\leq M$, we have 
\begin{equation}\label{les} 
\begin{split}
\sigma_k\big(\{a\in&\R^k:\,\log|\langle a,u\rangle|<-R\}\big)\leq\sigma_k\big(\{a\in\R^k:\,\log|\langle a,s\rangle|<-R\}\big)\\ 
& =\frac{1}{|s_1|}\int_{\R^{k-1}}\int_{\{|x_1|<e^{-R}\}} 
\phi\left(\frac{x_1-\sum_{j=2}^ks_jx_j}{s_1}\right)\phi(x_2)\ldots\phi(x_k)\,dx_1\ldots dx_k\\
&\leq 2\sqrt{2k}\,Me^{-R}\,.
\end{split}
\end{equation} 
If $R_0\geq\log k$, using \eqref{ues} and \eqref{les}, we obtain
\begin{align*}
\int_{\R^k}\big|\log|\langle a,u\rangle|\big|^\nu&\,d\sigma_k(a)=
\nu\int_0^{+\infty}R^{\nu-1}\sigma_k\big(\{a\in\R^k:\,\big|\log|\langle a,u\rangle|\big|>R\}\big)\,dR\\
&\leq\nu\int_0^{R_0} R^{\nu-1}\,dR + \nu\int_{R_0}^{+\infty} R^{\nu-1}\sigma_k\big(\{a\in\R^k: \,\big|\log|\langle a,u\rangle|\big|>R\}\big)\,dR\\
&\leq R_0^\nu + \nu\int_{R_0}^{+\infty} R^{\nu-1}\left(\frac{2^\rho ck}{R^\rho}+2\sqrt{2k}\,Me^{-R}\right)dR\,.
\end{align*}
Since $R^{\nu-1}e^{-R/2}\leq(2(\nu-1)/e)^{\nu-1}\leq\nu^{\nu-1}$ for $R\geq0$, and since $R_0\geq\log k$, we get 
\begin{align*}
J_k(u;\sigma_k,\nu)&\leq R_0^\nu+
\frac{2^\rho\nu ckR_0^{\nu-\rho}}{\rho-\nu}+
2\sqrt{2k}\,M\nu^\nu\int_{R_0}^\infty e^{-R/2}\,dR\\
&\leq R_0^\nu\left(1+\frac{2^\rho\nu ck}{(\rho-\nu)R_0^\rho}\right)+
4\sqrt{2}\,M\nu^\nu.
\end{align*}
The proposition follows by taking $R_0^\rho=k$. 
\end{proof}

We remark that hypothesis \eqref{tailc} holds for every 
$\rho>1$ in the case when the function $\phi$ has compact support. 
Particularly interesting is the case of the uniform distribution 
on the interval $[0,1]$, i.e. when $\phi=\chi_{[0,1]}$ 
is the characteristic function of $[0,1]$. 
In this case we obtain a much better bound for the moments 
$J_k(u;\sigma_k,\nu)$:

\begin{Proposition}\label{P:uniform}
Let $\sigma_k$ be the measure on $\R^k\subset\C^k$ defined in 
\eqref{e:riid} with $\phi=\chi_{[0,1]}$, so  
$d\sigma_k=\chi_{[0,1]^k}\,d\lambda_k$. For every integer $k\geq1$, 
every $\nu\geq1$, and every $u\in\C^k$ with $|u|=1$, we have 
\[J_k(u;\sigma_k,\nu)=
\int_{[0,1]^k}\big|\log|\langle a,u\rangle|\big|^\nu\,d\lambda_k(a)
\leq (\log k)^\nu+6\nu^\nu\,.\]
\end{Proposition}

\begin{proof} We proceed as in the proof of Proposition \ref{P:riid}. Since 
\[\left\{a\in\R^k:\,|\langle a,u\rangle|>e^R\right\}
\subset \bigcup_{j=1}^k\big\{a_j\in\R:\,|a_j|>e^{R-\frac12\log k}\big\},\]  
we have that $\sigma_k\big(\{a\in\R^k:\,\log|\langle a,u\rangle|>R\}\big)=0$ 
for $R\geq\frac12\,\log k$. Moreover, by \eqref{les}, 
\[\sigma_k\big(\{a\in\R^k:\,\log|\langle a,u\rangle|<-R\}\big)
\leq 2\sqrt{2k}\,e^{-R}\,.\]
It follows that 
\begin{align*}
J_k(u;\sigma_k,\nu)&=
\nu\int_0^{+\infty}R^{\nu-1}\sigma_k
\big(\{a\in\R^k:\,\big|\log|\langle a,u\rangle|\big|>R\}\big)\,dR\\
&\leq\nu\int_0^{\log k} R^{\nu-1}\,dR + 
\nu\int_{\log k}^{+\infty} R^{\nu-1}\sigma_k
\big(\{a\in\R^k: \,\log|\langle a,u\rangle|<-R\}\big)\,dR\\
&\leq(\log k)^\nu + 2\sqrt{2k}\,
\nu\int_{\log k}^{+\infty} R^{\nu-1}e^{-R}\,dR\\
&\leq(\log k)^\nu+2\sqrt{2k}\,
\nu^\nu\int_{\log k}^\infty e^{-R/2}\,dR=
(\log k)^\nu+4\sqrt{2}\,\nu^\nu\,.
\end{align*}
\end{proof}

%---------------------------------
\section{Visualization of expected distributions}\label{S:pictures}
In this section we illustrate some equidistribution results 
by visualizing the convergence of the expected measures 
\(p^{-1}\mathbb{E}[Z_{f_p}]\) for \(p\to \infty\) in two explicit examples. 

\subsection{Random polynomials with Gaussian coefficients}
We consider here the space \((\mathbb{C}_p[z],(\cdot,\cdot))\) 
which consists of polynomials in one complex variable of degree at most \(p\). 
The inner product \((\cdot,\cdot)\) is given by the integration over 
the square 
\(Q=[-\frac{1}{2},\frac{1}{2}]\times[-\frac{1}{2},\frac{1}{2}]
\subset \mathbb{C}\), that is
\[(f,g)=\int_{Q}f(z)\overline{g(z)}\,d\lambda_2(z),
\,\,\,f,g\in \mathbb{C}_p[z]\,.\]

Given an orthonormal basis \(s_0,\ldots,s_p\) for this space, 
our purpose is to understand the expected distribution of zeros 
for random polynomials \(f_p\) of the form
\[f_p(z)=a_0s_0(z)+\ldots+a_ps_p(z)\,,\]
where \(a_0,\ldots,a_p\) are i.i.d. complex Gaussian random variables 
with mean zero and variance one. Theorem \ref{BL2} applies in this 
setting with $Y=Q$ and $\varphi=0$. Since the equilibrium measure 
$dd^cV_{Q,0}$ is supported on the boundary of the square \(\partial Q\), 
Theorem \ref{BL2} implies that the zeros of \(f_p\) are expected to 
distribute towards \(\partial Q\) when \(p\) becomes large. 
Drawing the expected distribution for various degrees \(p\) 
leads to the pictures in Figure \ref{fig:g1}.

\begin{figure}[H]
	\begin{center}
	\includegraphics[scale=0.4]{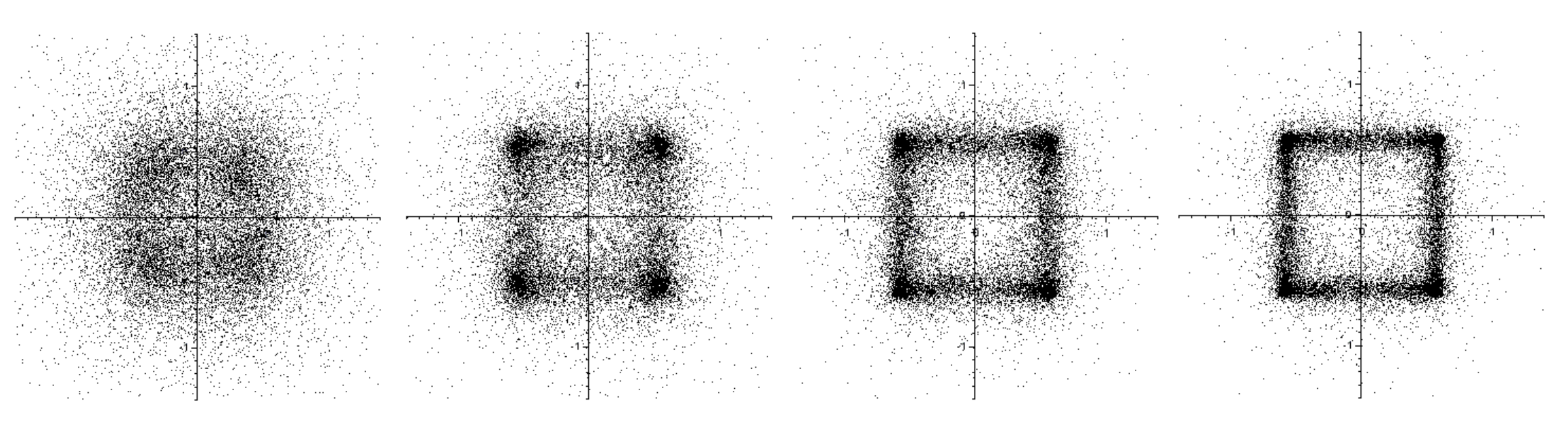}
	\end{center}
	\begin{flushleft}
		\small{(Made with Scilab, \textit{www.scilab.org}) 
		Expected distributions for zeros of random polynomials 
		\(f_p\) of degree \(p\) visualized using the law of large numbers. 
		From the left to the right depicted are the zeros contained in 
		\([-2,2]\times [-2,2]\subset\mathbb{C}\) of \(5000\) 
		random polynomials of degree 4, 1667 of degree \(12\), 
		833 of degree 24, and 500 of degree 40. 
		The polynomials \(f_p\) are randomly chosen with respect to 
		the Gaussian distribution on \((\mathbb{C}_p[z],(\cdot,\cdot))\). 
		The pictures were developed during a joint project with 
		Gerrit Herrmann (Regensburg).}  
	\end{flushleft} 
	\begin{center}
		\caption{Zero distributions for the square}\label{fig:g1}
	\end{center}
\end{figure}

To obtain the pictures we use the following strategy. 
Using the law of large numbers we can approximate 
the expected distribution by choosing a sequence of 
random polynomials of fixed degree \(p\), that is
\[\frac{1}{2\pi N p}\,\Delta\sum_{m=1}^{N}\log|f^{(m)}_p|\to 
p^{-1}\mathbb{E}(Z_{f_p}),\,\,\, N\to\infty.\]
By drawing the zeros of each of these polynomials as points 
in the plane we just illustrate the support of the measure 
\(\Delta\sum_{m=1}^{N}\log|f^{(m)}_p|\). Therefore we choose \(N\) 
depending on \(p\) such that \(Np\) is constant. 
Moreover, the fact that we can take $N$ small (and even $N=1$) 
for $p$ large is consistent with the second assertion of 
Theorem \ref{BL2} (see \eqref{asdist}).

Let us explain now how to construct the random polynomials \(f_p\). 
For \(p\) fixed we compute a positive definite hermitian 
\((p+1)\times (p+1)\) matrix \(S=(s_{lj})_{0 \leq l,j\leq p }\) 
with entries \(s_{lj}=(z^l,z^j)\). Applying the Cholesky decomposition to 
\(S\) we find (after inversion) a matrix \(R=(r_{lj})_{0 \leq l,j\leq p }\) 
with \(\operatorname{Id}=R^*SR\). Hence, we have that \(s_0,\ldots,s_p\) 
is an orthonormal Basis where \(s_j\) is given by
\[s_j(z)=\sum_{l=0}^{p}r_{lj}z^l,\,\,0\leq j\leq p.\]
Choosing Gaussian distributed  (pseudo) random numbers 
\(a_0,a_1,\ldots,a_p\) leads to the desired random polynomial 
\(f_p\).

\subsection{Random $SU_2$-polynomials with uniformly distributed coefficients}
We consider random \(SU_2\)-polynomials $f_p$ of degree \(p\), that is
\[f_p(z)=\sum_{j=0}^{p}a_j\sqrt{(p+1)\binom{p}{j}}\,z^j\]
where \(a_0,\ldots,a_p\) are i.i.d.\ real random numbers 
uniformly distributed on the interval \([0,1]\). 
Setting \(s_j(z)=\sqrt{(p+1)\binom{p}{j}}z^j\), \(0\leq j\leq p\) 
we have that \(s_0,\ldots,s_p\) is an orthonormal basis for the 
space \((\mathbb{C}_p[z],(\cdot,\cdot)_p)\) consisting of 
polynomials in one complex variable of degree at most \(p\) 
with inner product \((\cdot,\cdot)_p\) given by
\[(f,g)_p=\int_{\C}\frac{f(z)
\overline{g(z)}}{\pi(1+|z|^2)^{p+2}}\,d\lambda_2(z),\,\,\, 
f,g\in \C_p[z].\]
Theorem \ref{th1} applies in this case with the measures 
$\sigma_{p+1}$ from Proposition \ref{P:uniform}. 
It implies that for \(p\to\infty\) the zeros of \(f_p\) 
distribute to the measure 
\[d\mu(z)=\frac{1}{\pi(1+|z|^2)^{2}}\,d\lambda_2(z).\]
The measure \(d\mu\) is visualized in Figure~\ref{fig:g2}. 
Proceeding as in Example I we obtain pictures for the expected 
distributions of zeros \(p^{-1}\mathbb{E}(Z_{f_p})\) 
for various \(p\) (see Figure~\ref{fig:g3}). 
The computations in this example are easier since an orthonormal 
basis for \((\mathbb{C}_p[z],(\cdot,\cdot)_p)\) is already given. 

\begin{figure}[t]
	\begin{center}
\includegraphics[scale=0.5]{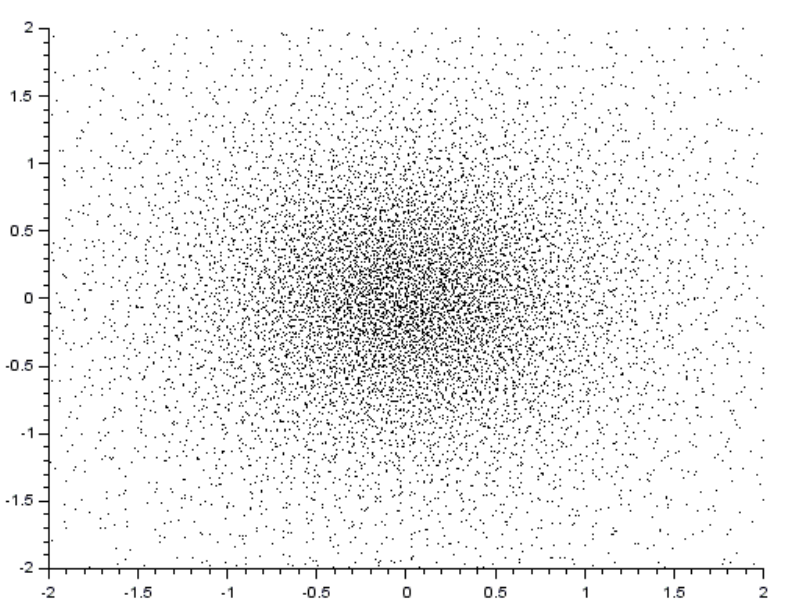}
	\end{center}
	\begin{flushleft}
		\small{(Made with Scilab, \textit{www.scilab.org}) 
		All points, which lie in the domain \([-2,2]\times [-2,2]\subset\mathbb{C}\), 
		out of 10240 randomly chosen points in \(\mathbb{C}\) 
		with respect to the distribution \(d\mu\).} 
	\end{flushleft} 
	\begin{center}
		\caption{Visualization of \(d\mu\)}\label{fig:g2}
	\end{center}
\end{figure}
In Figure~\ref{fig:g3}, starting with random polynomials 
of degree \(p=1\), we see that all the zeros lie on the 
negative real axis. That is clear since polynomials of 
the form \(z\mapsto a_0+a_1z\) with \(a_0,a_1\in\mathbb{R}_+\) 
have one zero given by \(z=-\frac{a_0}{a_1}\in\mathbb{R}_-\). 
For degree \(p=2\) we see that all the zeros have negative real 
part which immediately follows from solving the equation 
\(a_0+a_1z+a_2z^2=0\) with \(a_0,a_1,a_2> 0\). 
We see that the zeros of \(f_p\) avoid a region 
around the positive real axis which shrinks as $p$ 
becomes larger. Comparing Figure~\ref{fig:g3} 
to Figure~\ref{fig:g2} we see that the zeros of 
$f_p$ distribute to the measure $\nu$ as $p$ increases.

The pictures and all computations were made using Scilab 
(\textit{www.scilab.org}). It is important to mention that 
all the numerical computations and approximations generate 
some errors which we do not specify here. Hence, we do not 
claim the visualization of the true expected distributions. 

\begin{figure}[H]
	\begin{center}
\includegraphics[scale=0.29]{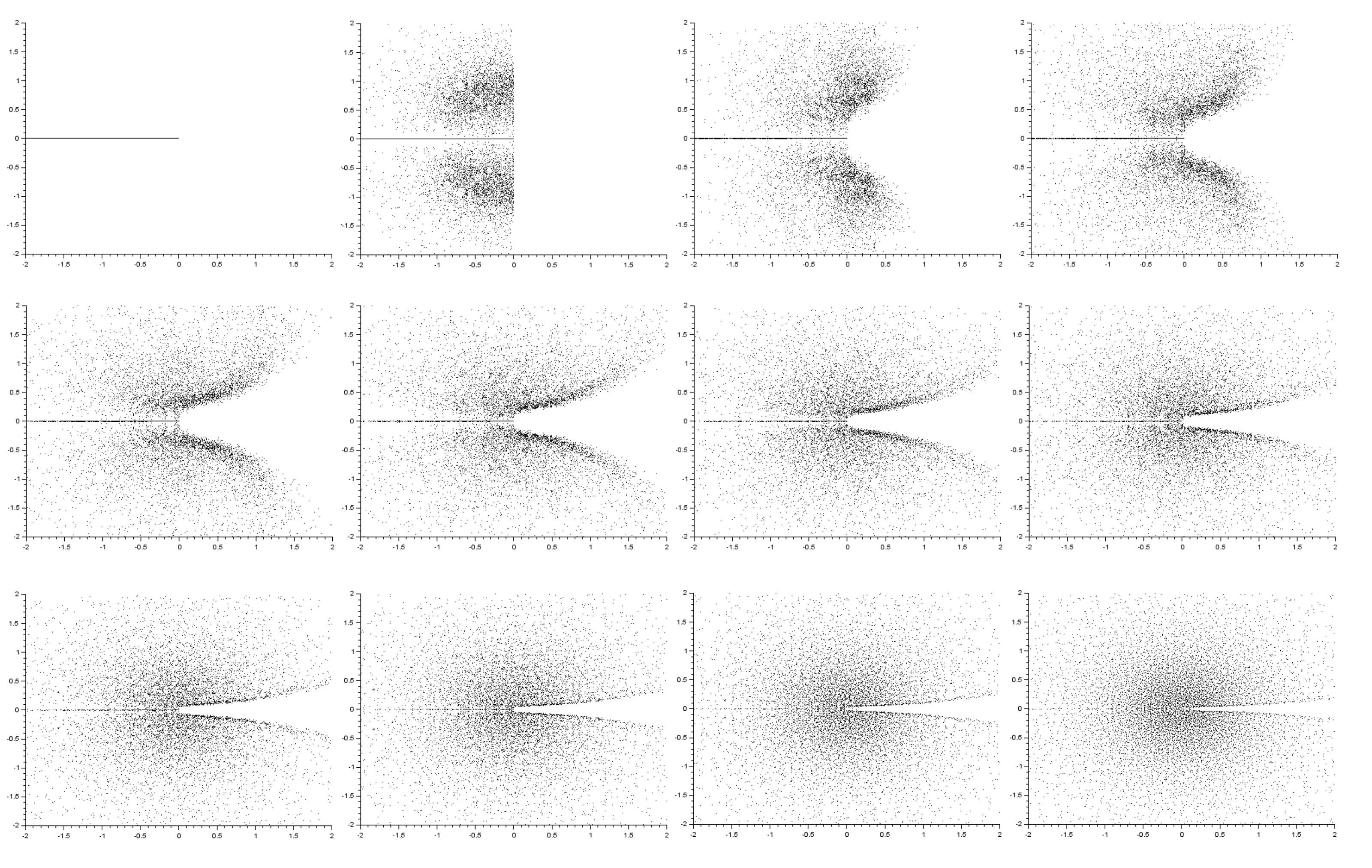}
	\end{center}
	\begin{flushleft}
		\small{(Made with Scilab, \textit{www.scilab.org}) 
		Expected distributions for zeros of random polynomials 
		\(f_p\) of degree \(p\) visualized using the law of large numbers. 
		Depicted are the zeros in the domain \([-2,2]\times [-2,2]\) of 
		\(5\cdot2^{13-j}\) random polynomials of degree \(p=2^{j-1}\) 
		with \(j=1,2,\ldots,12\), starting from the upper left down to 
		the bottom right corner. The polynomials \(f_p\) are random 
		\(SU_2\)-polynomials with real coefficients uniformly 
		distributed in the interval \([0,1]\).}
	\end{flushleft}
	\begin{center}
		\caption{Zero distribution for \(SU_2\)-polynomials}
		\label{fig:g3}
	\end{center}
\end{figure}

%---------------------------------

%%%%%%%%%%%%%%%%%

\end{document}